\documentclass[envcountsame,runningheads,11pt]{llncs}
\usepackage{a4wide}
\usepackage{xspace}
\usepackage{amsmath}
\usepackage{amssymb}
\usepackage{graphicx}
\usepackage{stmaryrd}
\newcommand{\myurl}[1]{\url{{\rm\texttt{#1}}}\xspace}
\newcommand{\IR}{\mathbb{R}}

\newcommand{\IQ}{\mathbb{Q}}

\newcommand{\IN}{\mathbb{N}}
\newcommand{\dom}{\operatorname{dom}}
\newcommand{\sgn}{\operatorname{sgn}}
\newcommand{\range}{\operatorname{range}}
\newcommand{\graph}{\operatorname{graph}}

\newcommand{\QH}{\operatorname{Q}_\text{H}}
\newcommand{\id}{\operatorname{id}}

\newcommand{\calO}{\mathcal{O}}
\newcommand{\calF}{\mathcal{F}}

\newcommand{\calM}{\mathcal{M}}
\newcommand{\calD}{\mathcal{D}}

\newcommand{\person}[1]{\textsc{#1}}

\newcommand{\mycite}[2]{{\rm\cite[\textsc{#1}]{#2}}}

\newcommand{\toto}{\rightrightarrows}

\newcommand{\ball}{B}
\newcommand{\closure}[1]{\overline{#1}}

\newcommand{\cball}{\closure{\ball}}
\newcommand{\myrho}{\rho}

\newcommand{\rhosd}{\myrho_{\text{sd}}}

\newcommand{\sdzero}{\textup{\texttt{0}}\xspace}
\newcommand{\sdone}{\textup{\texttt{1}}\xspace}
\newcommand{\sdminus}{\textup{\texttt{\={1}}}\xspace}
\newcommand{\sddot}{\textup{\texttt{.}}\xspace}
\newcommand{\myrhoC}{\myrho_{\text{C}}}

\newcommand{\Cantor}{\{0,1\}^\omega}
\newcommand{\COMMENTED}[1]{}

\spnewtheorem{observation}[theorem]{Observation}{\bfseries}{\itshape}
\spnewtheorem{fact}[theorem]{Fact}{\bfseries}{\itshape}
\spnewtheorem{myclaim}[theorem]{Claim}{\bfseries}{\itshape}
\spnewtheorem{scholium}[theorem]{Scholium}{\bfseries}{\itshape}
\spnewtheorem{myexample}[theorem]{Example}{\bfseries}{\itshape}
\spnewtheorem{myquestion}[theorem]{Question}{\bfseries}{\itshape}
\spnewtheorem{myremark}[theorem]{Remark}{\bfseries}{\itshape}
\spnewtheorem{myconjecture}[theorem]{Conjecture}{\bfseries}{\itshape}

\begin{document}
\setcounter{secnumdepth}{3}
\setcounter{tocdepth}{3}
\title{Relative Computability and Uniform Continuity of Relations}
\titlerunning{Relative Computability and Uniform Continuity of Relations}
\author{Arno Pauly\inst{1} \and Martin Ziegler\inst{2}}
\authorrunning{Arno Pauly \and Martin Ziegler}
\institute{Cambridge University \and Technische Universit\"{at} Darmstadt}
\date{}
\makeatletter
\renewcommand\maketitle{\newpage
  \refstepcounter{chapter}%
  \stepcounter{section}%
  \setcounter{section}{0}%
  \setcounter{subsection}{0}%
  \setcounter{figure}{0}
  \setcounter{table}{0}
  \setcounter{equation}{0}
  \setcounter{footnote}{0}%
  \begingroup
    \parindent=\z@
    \renewcommand\thefootnote{\@fnsymbol\c@footnote}%
    \if@twocolumn
      \ifnum \col@number=\@ne
        \@maketitle
      \else
        \twocolumn[\@maketitle]%
      \fi
    \else
      \newpage
      \global\@topnum\z@   
      \@maketitle
    \fi
    \thispagestyle{empty}\@thanks
    \def\\{\unskip\ \ignorespaces}\def\inst##1{\unskip{}}%
    \def\thanks##1{\unskip{}}\def\fnmsep{\unskip}%
    \instindent=\hsize
    \advance\instindent by-\headlineindent
    \if@runhead
       \if!\the\titlerunning!\else
         \edef\@title{\the\titlerunning}%
       \fi
       \global\setbox\titrun=\hbox{\small\rm\unboldmath\ignorespaces\@title}%
       \ifdim\wd\titrun>\instindent
          \typeout{Title too long for running head. Please supply}%
          \typeout{a shorter form with \string\titlerunning\space prior to
                   \string\maketitle}%
          \global\setbox\titrun=\hbox{\small\rm
          Title Suppressed Due to Excessive Length}%
       \fi
       \xdef\@title{\copy\titrun}%
    \fi
    \if!\the\tocauthor!\relax
      {\def\and{\noexpand\protect\noexpand\and}%
      \protected@xdef\toc@uthor{\@author}}%
    \else
      \def\\{\noexpand\protect\noexpand\newline}%
      \protected@xdef\scratch{\the\tocauthor}%
      \protected@xdef\toc@uthor{\scratch}%
    \fi
    \if@runhead
       \if!\the\authorrunning!
         \value{@inst}=\value{@auth}%
         \setcounter{@auth}{1}%
       \else
         \edef\@author{\the\authorrunning}%
       \fi
       \global\setbox\authrun=\hbox{\small\unboldmath\@author\unskip}%
       \ifdim\wd\authrun>\instindent
          \typeout{Names of authors too long for running head. Please supply}%
          \typeout{a shorter form with \string\authorrunning\space prior to
                   \string\maketitle}%
          \global\setbox\authrun=\hbox{\small\rm
          Authors Suppressed Due to Excessive Length}%
       \fi
       \xdef\@author{\copy\authrun}%
       \markboth{\@author}{\@title}%
     \fi
  \endgroup
  \setcounter{footnote}{\fnnstart}%
  \clearheadinfo}
\makeatother

\maketitle
\def\thefootnote{\fnsymbol{footnote}}
\addtocounter{footnote}{3}
\begin{abstract}
A type-2 computable real function is necessarily continuous;
and this remains true for relative, i.e. oracle-based computations.
Conversely, by the Weierstrass Approximation Theorem, every continuous
$f:[0,1]\to\mathbb{R}$ is computable relative to some oracle. 

In their search for a similar topological characterization of 
relatively computable \emph{multi-}valued functions 
$f:[0,1]\rightrightarrows\mathbb{R}$ (aka relations), 
Brattka and Hertling (1994) have considered
two notions: weak continuity (which is weaker than relative computability)
and strong continuity (which is stronger than relative computability).
Observing that \emph{uniform} continuity plays a crucial role
in the Weierstrass Theorem, we propose and compare several notions
of uniform continuity for relations. Here, due to the additional 
quantification over values $y\in f(x)$, new ways arise of (linearly) 
ordering quantifiers---yet none turns out as satisfactory.

We are thus led to a notion of uniform continuity based on
the \textsf{Henkin quantifier}; and prove it necessary
for relative computability of compact real relations. 
In fact iterating this condition
yields a strict hierarchy of notions each necessary,
and the $\omega$-th level also sufficient,
for relative computability.
\end{abstract}
\begin{center}
\begin{minipage}[c]{0.9\textwidth}\vspace*{-8ex}%
\renewcommand{\contentsname}{}
\tableofcontents
\end{minipage}
\end{center}
\smallskip
\section{Introduction}
A simple counting argument shows that not every (total) integer 
function $f:\IN\to\IN$ can be computable; 
on the other hand, each such function
can be encoded into an oracle $\calO\subseteq\{0,1\}^*$
that renders it relatively computable.
Over real numbers, similarly, not every total $f:[0,1]\to\IR$
can be computable for cardinality reasons; 
and this remains true for oracle machines.
In fact it is folklore in Recursive Analysis that any function $f$
computably mapping approximations of real numbers $x$ to approximations of $f(x)$
must necessarily be continuous; and the same remains true for oracle
computations. Even more surprisingly, this implication can be reversed: 
If a (say, real) function $f$ is continuous, then there exists an oracle 
which renders $f$ computable\footnote{%
It has been observed that a continuous function $f:[0,1]\to[0,1]$ will usually
not have a \emph{least} oracle rendering it computable \cite{Miller}}. 
This can for instance be concluded from 
the Weierstrass Approximation Theorem. 
A far reaching generalization from the reals
to so-called \emph{admissibly represented spaces} is the 
\textsf{Kreitz-Weihrauch Theorem}, 
cf. e.g. \mycite{3.2.11}{Weihrauch}
and compare the \textsf{Myhill-Shepherdson Theorem} 
in Domain Theory.
The equivalence between continuity and relative computability has
led \person{Dana Scott} to consider continuity as an approximation
to computability.

Now many computational problems are more naturally expressed as 
relations (i.e. multivalued) rather than as (single-valued) functions.
For instance when diagonalizing a given real symmetric matrix, one
is interested in \emph{some} basis of eigenvectors, not a specific one.
It is thus natural to consider computations which, given $x$, 
intensionally choose and output \emph{some} value $y\in f(x)$.
Indeed, a multifunction may well be computable 
yet admit no continuous single-valued \emph{selection};
cf. e.g. \mycite{Exercise~5.1.13}{Weihrauch} or \cite{Luckhardt}.
Hence multivaluedness avoids some of the topological restrictions 
of single-valued functions---but of course not all of them.
Specifically it is easy to see that a multifunction $f$ is relatively computable
~iff~ it admits a continuous so-called \emph{realizer},
that is a function mapping any infinite binary string encoding some $x$
to an infinite binary string encoding some $y\in f(x)$.

However the single-valued case raises the hope for an intrinsic
characterization of relative computability of $f$, without
referral to Cantor space. Such an investigation has been pursued
in \cite{VascoPeter94}, yielding both necessary and sufficient
conditions for a relation to be computable relative to some oracle
(which, there, is called \emph{relative continuity} and we 
shall denote as \emph{relative computability}).
\person{Brattka} and \person{Hertling} 
have established what remains to-date the best counterpart to the
Kreitz-Weihrauch Theorem for the multivalued case:

\begin{fact} \label{f:VascoPeter94}
Let $X,Y$ be separable metric spaces and $Y$ in addition complete.
Then a pointwise closed relation $f:X\toto Y$ 
is relatively computable
~iff~ it has a strongly continuous tightening\footnote{We
reserve the original term ``restriction'' to denote either
$f|_{A}:=f\cap (A\times Y)$ or $f|^B:=f\cap (X\times B)$
for some $A\subseteq X$ or $B\subseteq Y$.}
\end{fact}
Here, being \emph{pointwise closed} means that
$f(x):=\{y\in Y:(x,y)\in f\}$ is a closed subset
for every $x\in X$. We shall freely switch between
the viewpoint of $f:\subseteq X\toto Y$ 
being a relation ($f\subseteq X\times Y$)
and being a set-valued partial mapping
$f:\subseteq X\to 2^Y$, $x\mapsto f(x)$.
Such $f$ is considered \emph{total}
(written $f:X\toto Y$)
if $\dom(f):=\{x\in X:f(x)\neq\emptyset\}$ 
coincides with $X$. 
Following \mycite{Definition~7}{Multifunctions},
$g$ is said to \emph{tighten} $f$ 
(and $f$ to \emph{loosen} $g$) 
if both $\dom(f)\subseteq\dom(g)$ and 
$\forall x\in\dom(f): g(x)\subseteq f(x)$ hold;
see Figure~\ref{f:tighten}a) and note
that tightening is obviously reflexive and transitive.
Furthermore write
$f[S]:=\bigcup_{x\in S} f(x)$ 
for $S\subseteq X$ and $\range(f):=f[X]$;
also $f|_S:=f\cap(S\times Y)$ and
$f|^T:=f\cap (X\times T)$ for $T\subseteq Y$.
Finally let $f^{-1}:=\{(y,x):(x,y)\in f\}$
denote the \emph{inverse} of $f$, i.e. 
such that $(f^{-1})^{-1}=f$ and
$\range(f)=\dom(f^{-1})$.

\begin{figure}[hbt]
\centerline{\includegraphics[width=0.6\textwidth]{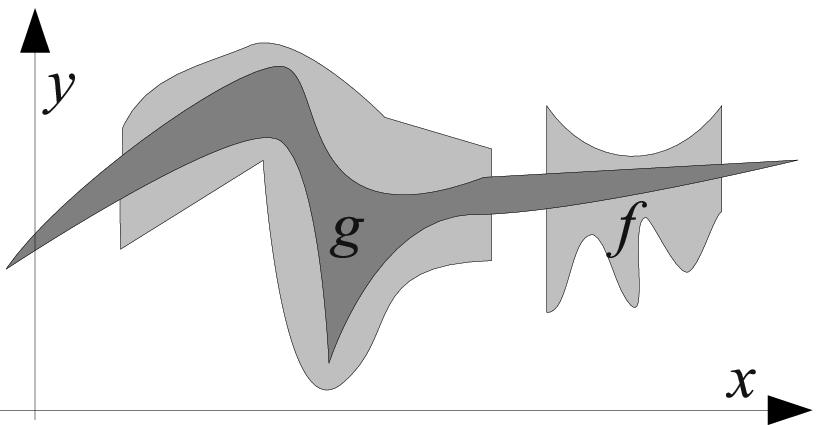}\quad
\smash{\raisebox{-3ex}{\includegraphics[width=0.28\textwidth]{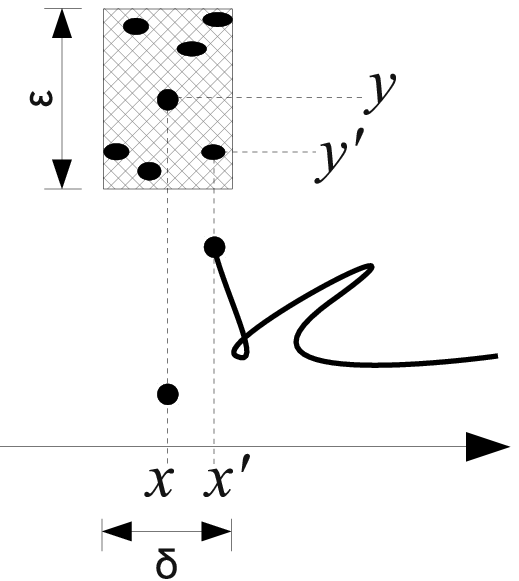}}}}
\caption{\label{f:tighten}%
a) For a relation $g$ (dark gray) to tighten $f$ (light gray) means
no more freedom (yet the possibility) to choose some $y\in g(x)$ 
than to choose some $y\in f(x)$ (whenever possible).
\quad b) Illustrating $\epsilon$--$\delta$--continuity in $(x,y)$
for a relation (black)}
\end{figure}

\section{Continuity for Relations}
For multivalued mappings, the literature knows a variety
of easily confusable notions of continuity
like \mycite{\textsection 7}{KleinThompson} or \cite{Neumaier}.
Some of them capture the intuition that, upon input $x$,
\emph{all} $y\in f(x)$ occur as output for \emph{some}
`nondeterministic' choice \mycite{Section~7}{BrattkaCooper};
or that the `value' $f(x)$ be produced extensionally as a set \cite{Spreen}.
Here we pursue the original conception
that, upon input $x$, \emph{some} value $y$ be output
subject to the condition $y\in f(x)$.

\begin{definition} \label{d:Main}
Let $(X,d)$ and $(Y,e)$ denote metric spaces and
abbreviate $\ball(x,r):=\{x'\in X:d(x,x')<r\}\subseteq X$
and $\cball(x,r):=\{x'\in X:d(x,x')\leq r\}$; similarly for $Y$.
\\
Now fix some $f:\subseteq X\toto Y$
and call $(x,y)\in f$  a \textsf{point of continuity of $f$}
if the following formula holds (cf. Figure~\ref{f:tighten}b):
\begin{equation*}
\forall \varepsilon>0 \;\; \exists \delta>0 \;\;
\forall x'\in\ball(x,\delta)\cap\dom(f) \;\;
\exists y'\in\ball(y,\varepsilon)\cap f(x') \enspace . 
\end{equation*}
\begin{enumerate}
\setlength{\abovedisplayskip}{0.4\baselineskip plus0.4\baselineskip}
\item[a)]
Call $f$ \textsf{strongly continuous} if every
$(x,y)\in f$ is a point of continuity of $f$; equivalently:
\begin{equation*} 
\forall x\in\dom(f) \;\; \forall y\in f(x) \;\;
\forall \varepsilon>0 \;\; \exists \delta>0 \;\;
\forall x'\in\ball(x,\delta)\cap\dom(f) \;\;
\exists y'\in\ball(y,\varepsilon)\cap f(x').
\end{equation*}
\item[b)]
Call $f$ \textsf{weakly continuous} if 
the following holds:
\begin{equation*} 
\forall x\in\dom(f) \;\; \exists y\in f(x) \;\;
\forall \varepsilon>0 \;\; \exists \delta>0 \;\;
\forall x'\in\ball(x,\delta)\cap\dom(f) \;\;
\exists y'\in\ball(y,\varepsilon)\cap f(x'). 
\end{equation*}
\item[c)]
Call $f$ \textsf{uniformly weakly continuous} if the following holds:
\begin{equation*} 
\forall \varepsilon>0 \;\; \exists \delta>0 \;\;
\forall x\in\dom(f) \;\; \exists y\in f(x) \;\;
\forall x'\in\ball(x,\delta)\cap\dom(f) \;\;
\exists y'\in\ball(y,\varepsilon)\cap f(x').
\end{equation*}
\item[d)]
Call $f$ \textsf{nonuniformly weakly continuous} if the following holds:
\begin{equation*} 
\forall \varepsilon>0 \;\; \forall x\in\dom(f) \;\; 
\exists \delta>0 \;\; \exists y\in f(x) \;\;
\forall x'\in\ball(x,\delta)\cap\dom(f) \;\;
\exists y'\in\ball(y,\varepsilon)\cap f(x').
\end{equation*}
\item[e)]
Call $f$ \textsf{uniformly strongly continuous} if the following holds:
\begin{equation*} 
\forall \varepsilon>0 \;\; \exists \delta>0 \;\;
\forall x\in\dom(f) \;\; \forall y\in f(x) \;\;
\forall x'\in\ball(x,\delta)\cap\dom(f) \;\;
\exists y'\in\ball(y,\varepsilon)\cap f(x').
\end{equation*}
\item[f)]
Call $f$ \textsf{semi-uniformly strongly continuous} if the following holds:
\begin{equation*} 
\forall \varepsilon>0 \;\; 
\forall x\in\dom(f) \;\; 
\exists \delta>0 \;\;
\forall y\in f(x) \;\;
\forall x'\in\ball(x,\delta)\cap\dom(f) \;\;
\exists y'\in\ball(y,\varepsilon)\cap f(x').
\end{equation*}
\end{enumerate}
\end{definition}
Items~a) and b) are
quoted from \mycite{Definition~2.1}{VascoPeter94}.
In the single-valued case, quantifications over $y\in f(x)$
and $y'\in f(x')$ drop out. Here, all a),b),d),f) collapse to
classical continuity; and both c) and e) to uniform continuity.
In the multivalued case, however, these notions are
easily seen distinct. 
Note for instance that in f), $\delta$ may depend on $x$ but not on $y$;
whereas $y$ may depend on $\varepsilon$ in c) but not in b). 
Logical connections between the various notions 
are collected in the following

\begin{lemma} \label{l:Implications}
\begin{enumerate}
\item[a)]
Strong continuity implies weak continuity
\item[b)]
but not vice versa.
\item[c)]
Weak continuity implies nonuniform weak continuity.
\item[d)]
Uniform weak continuity implies nonuniform weak continuity.
\item[e)]
Let $f$ be uniformly weakly continuous
and suppose that $f(x)\subseteq Y$ is compact for every $x\in X$.
Then $f$ is weakly continuous.
\item[f)]
Uniform strong continuity implies semi-uniform strong continuity 
\\ which in turn implies strong continuity.
\item[g)]
For compact $\dom(f)\subseteq X$, nonuniform weak continuity
implies uniform weak continuity.
\item[h)]
If $f(x)\subseteq Y$ is compact for every $x\in X$,\\
then strong continuity implies semi-uniform strong continuity.
\item[j)]
If $f\subseteq X\times Y$ is compact and strongly continuous,
it is uniformly strongly continuous.
\item[k)]
If $f\subseteq X\times Y$ is compact, then
so are $\dom(f)\subseteq X$ and $f[S]\subseteq Y$,
for every closed $S\subseteq X$; in particular $f(x)$ is compact.
\item[$\ell)$]
If $X$ is compact and single-valued total $f:X\to Y$ is continuous,
then both $f\subseteq X\times Y$ 
and its inverse $f^{-1}\subseteq Y\times X$ are compact.
\end{enumerate}
\end{lemma}
Note that the (classically trivial) implication from
(weak) uniform continuity to (weak) continuity in e)
is based on the (again, classically trivial) hypothesis
that $f(x)\subseteq Y$ be compact. Similarly, the classical 
fact that continuity on a compact set classically yields uniform 
continuity is generalized in g)+c).

\begin{figure}[htb]
\centerline{\includegraphics[width=0.99\textwidth]{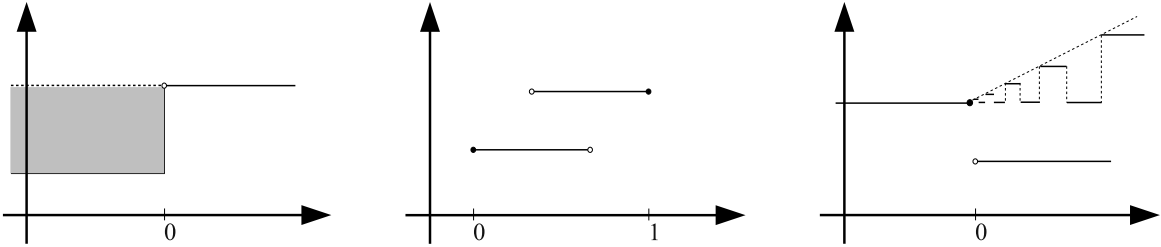}}
\caption{\label{f:multival}%
a) Example of a uniformly weakly continuous but not weakly continuous relation.
b) A semi-uniformly strongly continuous relation which is not uniformly strongly continuous.
c) A compact, weakly and uniformly weakly continuous relation 
which is not computable relative to any oracle.}
\end{figure}

\begin{proof} Items~a),c), d), and f) are obvious.
\begin{enumerate}
\item[b)]
is due to \mycite{Proposition~2.3(3)}{VascoPeter94};
cmp.~Example~\ref{x:Examples}d).
\item[e)]
Fix $x\in\dom(f)$. By hypothesis there exists,
to every $\varepsilon=1/n$, some $\delta_n$
and $y_n\in f(x)$ with:
$\forall x'\in\ball(x,\delta_n)\cap\dom(f)\;
\exists y'\in\ball(y_n,1/n)\cap f(x')$.
Now since $f(x)$ is compact, there some subsequence 
$y_{n_m}$ of $y_n$ converges to, say, $y_0\in f(x)$
with $d(y_{n_m},y_0)\leq1/m$.
We claim that this $y_0$ (which does not depend on $\varepsilon$ anymore)
satisfies
$$ \forall \varepsilon=2/m>0 \;\; \exists \delta:=\delta_{n_m}>0 \;\;
\forall x'\in\ball(x,\delta)\cap\dom(f) \;\;
\exists y'\in\ball(y_0,\varepsilon)\cap f(x'). $$
Indeed,
to arbitrary $x'\in\ball(x,\delta_{n_m})\cap\dom(f)$,
the hypothesis yields some $y'\in\ball(y,1/m)\cap f(x')$.
Then, by triangle inequality, it follows
$y'\in\ball(y_0,2/m)$. \\
Note that a different $x$ may require a different subsequence $n_m$;
hence $\delta$ may become dependent on $x$ even if it did not before.
\item[g)]
We claim that 
Definition~\ref{d:Main}d) is equivalent to the formula
\begin{equation} \label{e:Nonunif}
\forall \varepsilon>0 \;\; \forall x\in\dom(f) \;\;
\exists \delta>0: \quad \Phi(f,\varepsilon,x,\delta)
\end{equation}
where $\Phi(f,\varepsilon,x,\delta)$ abbreviates the predicate
\begin{equation*} 
\forall x'\in\ball(x,\delta)\cap\dom(f) \;\;
\exists y'\in f(x') \;\; 
\forall x''\in\ball(x,\delta)\cap\dom(f) \;\;
\exists y''\in f(x''):
\quad e(y',y'')<\varepsilon 
\end{equation*}
Indeed, $x',x''\in\ball(x,\delta)$ yield
$y'\in f(x')\cap\ball(y,\varepsilon)$ and
$y''\in f(x'')\cap\ball(y,\varepsilon)$,
hence $e(y',y'')<2\varepsilon$ by triangle inequality; 
and, conversely, $x':=x$ yields $y\in f(x)$.
Next observe that, again by triangle inequality, 
$\Phi(f,\varepsilon,x,\delta)$ implies
$\Phi(f,\varepsilon,z,\delta/2)$ for all
$z\in\ball(x,\delta/2)\cap\dom(f)$. 
Now for arbitrary but fixed $\varepsilon$
and to every $x\in\dom(f)$ there exists by hypothesis
some $0<\delta=\delta(x)$ such that
$\Phi\big(f,\varepsilon,x,\delta(x)\big)$ holds.
The open sets $\ball\big(x,\delta(x)/2\big)$ cover $\dom(f)$; 
and by compactness, finitely many of them suffice to do so: 
say, $\ball\big(x,\delta(x_i)/2\big)$, $i=1,\ldots,I$.
Now take $\bar\delta>0$ as the minimum over
these finitely many $\delta(x_i)/2$:
it will satisfy $\Phi(f,\varepsilon,\bar y,\bar\delta)$ 
for all $\bar y\in\dom(f)$.
\item[h)]
Similarly to g), consider the predicate
\begin{multline*}
\forall \varepsilon>0\;\;
\forall x\in\dom(f)\;\;
\forall y\in f(x)\;\;
\exists \delta\in(0,\varepsilon) \\
\underbrace{
\forall x',x''\in\ball(x,\delta)\cap\dom(f)\;\;
\forall y'\in f(x')\cap\ball(y,\delta)\;\;
\exists y''\in f(x'')\cap\ball(y',\varepsilon)}_{
=:\Phi(f,\varepsilon,x,y,\delta)}
\end{multline*}
and note that it is equivalent to strong continuity:
The restriction to $\delta<\varepsilon$ is no loss of generality;
$y'\in\ball(y,\delta)$ and $y''\in f(x'')\cap\ball(y,\varepsilon)$
according to b) implies $e(y',y'')<\delta+\varepsilon<2\varepsilon$ 
arbitrary; whereas, conversely, 
strong continuity is recovered with $x':=x$ and $y':=y$.
Finally, $\Phi(f,\varepsilon,x,y,\delta)$ implies
$\Phi(f,\varepsilon,x,\bar y,\delta/2)$ for all
$\bar y\in\ball(y,\delta/2)$. 
Now the balls $\ball\big(y,\delta(y)/2\big)$, $y\in f(x)$,
cover $f(x)$; and by compactness, finitely many of them
suffice to do so.
\item[j)]
This time abbreviate 
\[ \Phi(f,x,y,\varepsilon,\delta) \quad:=\quad
\forall x'\in\ball(x,\delta)\cap\dom(f)
\;\;\exists y'\in f(x')\cap\ball(y,\varepsilon) \]
and observe that strong continuity 
$\forall\varepsilon>0 \;\forall (x,y)\in f \;\exists\delta>0
\;\; \Phi(f,x,y,\varepsilon/2,\delta)$
is equivalent to 
$\forall\varepsilon>0 \;\forall (x,y)\in f \;\exists\delta>0
\;\; \Phi(f,x,y,\varepsilon,\delta/2)$.
Moreover, $\Phi(f,x,y,\varepsilon/2,\delta)$ and
$(\bar x,\bar y)\in f\cap\big(\ball(x,\delta/2)\times\ball(y,\varepsilon/2)\big)$
together imply $\Phi(f,\bar x,\bar y,\varepsilon,\delta/2)$.
For fixed $\varepsilon>0$ there exists by hypothesis to each
$(x,y)\in f$ some $\delta=\delta(x,y)$ such that 
$\Phi(f,x,y,\varepsilon/2,\delta)$.
The open balls $\ball\big(x,\delta(x,y)/2)\times\ball(y,\varepsilon/2)$,
$(x,y)\in f$,
thus cover $f$; and by compactness, already finitely many of them 
suffice to do so.
Taking $\bar\delta$ as the minimum of their corresponding
$\delta(x,y)$, we conclude
that $\Phi(f,x,y,\varepsilon,\bar\delta/2)$ holds for all $(x,y)\in f$:
uniform strong continuity.
\item[k)]
Let $U_i\subseteq X$ ($i\in I$) denote an open covering of $\dom(f)$.
Then $U_i\times Y$ is an open covering of $f$,
hence contains a finite subcover: whose projection
onto the first component is a finite subcover of $U_i$.
\\
Similarly, let $V_j\subseteq Y$ ($j\in J$)
denote an open covering
of $f[S]\subseteq Y$. Then $X\times V_j$, 
together with $(X\setminus S)\times Y$,
constitutes an open covering of $f$;
hence contains a finite subcover:
and the corresponding $V_j$ yield 
a finite subcover of $f[S]$.
\\
Finally, $S:=\{x\}$ is closed
and thus also $f[S]=f(x)$.
\item[$\ell)$]
Let $(x_n,y_n)\subseteq f$ be a sequence.
Since $(x_n)\subseteq X$ compact, it has a converging subsequence;
w.l.o.g. $(x_n)$ itself. Now by continuity and single-valuedness, 
$y_n=f(x_n)\to f(x)$ converges. Thus, $f$ is compact;
and homeomorphic to $f^{-1}$.
\qed\end{enumerate}\end{proof}
We say that $f$ is \textsf{pointwise compact} if 
$f(x)\subseteq Y$ is compact for every $x\in\dom(f)$.
Any single-valued $f$ automatically satisfies this condition;
which in turn implies being \emph{pointwise closed} 
as required in Fact~\ref{f:VascoPeter94}.
Pointwise compactness is essential for uniform weak continuity 
to imply weak continuity in Lemma~\ref{l:Implications}e):

\begin{myexample} \label{x:Examples}
\begin{enumerate}
\item[a)]
The multifunction from {\rm\cite[\textsc{Example}~27c]{DiscreteAdvice}},
namely
\[ f:[-1,+1]\toto[0,1], \qquad
0\geq x\mapsto [0,1), \quad 0<x\mapsto\{1\} \]
depicted in Figure~\ref{f:multival}a),
is uniformly weakly continuous but not weakly continuous. 
\item[b)]
The multifunction $g:[0,1]\toto[0,1]$ with
$\graph(g)=\big([0,2/3)\times\{0\}\big)\cup\big((1/3,1]\times\{1\}\big)$
depicted in Figure~\ref{f:multival}b)
has compact $\dom(g)$ and $g(x)$ for every $x$
but $\graph(g)$ is not compact.
Moreover, $g$ is semi-uniformly strongly continuous
but not uniformly strongly continuous.
\item[c)]
The relation 
$\big(\IQ\times(\IR\setminus\IQ)\big)\cup\big((\IR\setminus\IQ)\times\IQ\big)$
from \mycite{Example~7.2}{VascoPeter94} is uniformly strongly continuous.
\item[d)]
Inspired by \mycite{Proposition~2.3(3)}{VascoPeter94},
the relation $g:[-1,+1]\toto[-1,+1]$
depicted in Figure~\ref{f:multival}c) with graph
\begin{equation} \label{e:VascoPeter94}
 \{(x,0):x\leq0\}\;\cup\;\{(x,-1):x>0\}\;
\cup\;\big\{\big(x,\tfrac{1+(-1)^n}{n+1}\big):n\in\IN,1/(n+1)\leq x\leq 1/n\big\}
\end{equation}
is compact and both weakly continuous 
and uniformly weakly continuous
but not strongly continuous.
\end{enumerate}
\end{myexample}
\begin{proof}
\begin{enumerate}
\item[a)]
To assert uniform weak continuity, consider
$\delta=\delta(\varepsilon):=\varepsilon$. Moreover
let $y=y(x,\varepsilon):=1$ for $x>0$ and
$y(x,\varepsilon):=1-\varepsilon/2$ for $x\leq0$.
Then, in case $x'>0$, choose $y':=1$;
and in case $x'\leq0$, chose $y':=1-\varepsilon/2$.
\\
Suppose $f$ is weakly continuous at $x:=0$, i.e.
there exists some appropriate $y\in f(x)=[0,1)$.
The consider $\varepsilon:=1-y$ and the induced
$\delta>0$ as well as $x':=\delta/2$:
No $y'\in f(x')=\{1\}$ can satisfy 
$\varepsilon>|y'-y|=1-y$, contradiction.
\item[b)]
Note $\dom(g)=[0,1]$ and $g(x)=\{0\}$
for $x\leq1/3$, $g(x)=\{0,1\}$ for
$1/3<x<2/3$, and $g(x)=\{1\}$ for
$x\geq2/3$: all compact.
Concerning semi-uniform strong continuity,
for $x\leq1/3$ let $\delta:=1/3$ and $y':=0=y$;
for $x\geq2/3$ let $\delta:=1/3$ and $y':=1=y$;
whereas for $1/3<x<2/3$, choose 
$\delta:=\min(2/3-x,x-1/3)$ and $y':=y$.
Uniform strong continuity leads to a
contradiction when considering 
$x:=1/3+\delta/2$ and $y:=1$
and $x':=1/3$.
\item[c)]
Let $\delta:=1$; 
then observe that $\IQ$ is dense in $\IR\setminus\IQ$
and $\IR\setminus\IQ$ is dense in $\IQ$.
\item[d)] 
Concerning weak continuity, 
in case $x\leq0$ choose $y:=0$ and $\delta:=\varepsilon$:
then, to $x'\in\ball(x,\delta)$,
$y':=0$ will do for $x'\leq 0$ as well as for
every $x'\in[1/(n+1),1/n]$ with $n$ odd;
and $y':=2/(n+1)$ for $x'\in[1/(n+1),1/n]$ with even $n$.
In case $x>0$ choose $y:=-1$ and $\delta:=x$;
then $x'\in\ball(x,\delta)$
implies $x'>0$ and $y':=-1$ works.
\\
Regarding uniform weak continuity,
let $\delta:=\varepsilon$ and distinguish 
cases $x<\varepsilon$ and $x\geq\varepsilon$.
In the former case, 
$y:=0$ will do for $x\leq0$ and for $x\in[1/(n+1),1/n]$ with $n$ odd;
and $y:=2/(n+1)$ for $x\in(0,\varepsilon)\cap[1/(n+1),1/n]$ with even $n$.
In the latter case, $y:=-1$ works.
\\
Strong continuity is violated, e.g., at
$(x,y)=(1/2,2/3)$ for $\varepsilon:=1/4$.
\qed\end{enumerate}\end{proof}

\subsection[\ldots and Computability of Relations]{Continuity and Computability of Relations}
Recall 
that (relative) computability of a multifunction $f:\subseteq\IR\toto\IR$
means that some (oracle) Turing machine can, 
upon input of any sequence of integer fractions $a_n/b_n$
with $|x-a_n/b_n|\leq2^{-n}$ for every $n\in\IN$ and some $x\in\dom(f)$,
output a sequence $u_m/v_m$ of integer fractions with
$|y-u_m/v_m|\leq2^{-m}$ for every $m\in\IN$ and some $y\in f(x)$.
More generally, 
a multifunction $f:\subseteq A\toto B$ between represented spaces $(A,\alpha)$
and $(B,\beta)$ is considered (relatively) computable if it admits a 
(relatively) computable $(\alpha,\beta)$--realizer,
that is a function $F:\subseteq\Cantor\to\Cantor$ 
mapping every $\alpha$--name of some $a\in\dom(a)$ to
a $\beta$--name of some $b\in f(a)$
\mycite{Definition~3.1.3}{Weihrauch}.

\begin{lemma} \label{l:Tighten}
Define the composition of multifunction $f:\subseteq X\toto Y$
and $g:\subseteq Y\toto Z$ as
\begin{equation} \label{e:Composition}
 g\circ f  \;:=\;
  \big\{(x,z) \big| x\in X, z\in Z, f(x)\subseteq\dom(g), 
  \;\exists y\in Y: (x,y)\in f\wedge (y,z)\in g \} \enspace .
\end{equation}
\begin{enumerate}
\item[a)]
$\id_X$ tightens $f^{-1}\circ f$; 
if $f$ is single-valued, then $f\circ f^{-1}=\id_{\range(f)}$.
\item[b)]
If $f'$ tightens $f$ and $g'$ tightens $g$,
then $g'\circ f'$ tightens $g\circ f$.
\item[c)]
If $\range(f)\subseteq\dom(g)$ holds
and both $f$ and $g$ are compact,
then so is $g\circ f$.
\item[d)]
If $\range(f)\subseteq\dom(g)$ holds
and if both $f$ and $g$ map compact sets to compact sets,
then so does $g\circ f$.
\item[e)]
Fix representations $\alpha$ for $X$ and $\beta$ for $Y$.
A multifunction $F:\subseteq\Cantor\toto\Cantor$ 
tightens $\beta^{-1}\circ f\circ\alpha$ ~iff~
$\beta\circ F\circ\alpha^{-1}$ tightens $f$.
\item[f)]
A function $F:\subseteq\Cantor\to\Cantor$ 
is an $(\alpha,\beta)$--realizer of $f$ ~iff~
$F$ tightens $\beta^{-1}\circ f\circ\alpha$ ~iff~
$\beta\circ F\circ\alpha^{-1}$ tightens $f$.
\end{enumerate}
\end{lemma}
Motivated by f),
let us call a multifunction $F$ as in e) an
$(\alpha,\beta)$--\emph{multirealizer} of $f$.
\begin{proof}
\begin{enumerate}
\item[a)]
Note $f(x)\subseteq\dom(f^{-1})$ and
$f^{-1}\circ f=\{(x,x'):\exists y: (x,y),(x',y)\in f\}$.
\item[b)]
Note $\dom(g\circ f)=\{x\in\dom(f):f(x)\subseteq\dom(g)\}$;
hence $\dom(f)\subseteq\dom(f')\wedge\dom(g)\subseteq\dom(g')\wedge f(x)\supseteq f'(x)
\wedge g(y)\supseteq g'(y)$ implies $\dom(g\circ f)\subseteq\dom(g'\circ f')$
 as well as 
 $\big(g'\circ g'\big)(x)=\{z:\exists y\in f'(x): z\in g'(y)\}\subseteq\big(g\circ f\big)(x)$;
 cmp. \mycite{Lemma~8.3}{Multifunctions}.
\item[c)]
Since $\range(f)\subseteq\dom(g)$, 
$g\circ f$ is the image of compact 
$\big(f\times\range(g)\big)\cap\big(\dom(f)\times g\big)\subseteq X\times Y\times Z$
under the continuous projection 
$\Pi_{1,3}:X\times Y\times Z\ni (x,y,z)\mapsto(x,z)\in X\times Z$.
\item[d)]
immediate from $\big(g\circ f\big)[S]=g\big[f[S]\big]$,
holding under the hypothesis $\range(f)\subseteq\dom(g)$.
\item[e)]
If $F$ tightens $\beta^{-1}\circ f\circ\alpha$,
then $\beta\circ F\circ\alpha^{-1}$
tightens $\beta\circ\beta^{-1}\circ f\circ\alpha\circ\alpha^{-1}$
due to b); which in turn coincides with
$\id_X\circ f\circ\id_Y=f$ according to a). \\
Conversely, $F=\id_{\Cantor}\circ F\circ\id_{\Cantor}$
tightens $\beta\circ\beta^{-1}\circ F\circ\alpha\circ\alpha^{-1}$
by a); which in turn tightens $\beta^{-1}\circ f\circ\alpha$
by hypothesis and by b).
\item[f)]
$F$ being an $(\alpha,\beta)$--realizer of $f$ means
$\dom(F)\supseteq\dom(f\circ\alpha)$ and 
$\beta\big(F(\bar\sigma)\big)\in f\big(\alpha(\bar\sigma)\big)$
for every $\bar\sigma\in\dom(f\circ\alpha)=\dom(\beta^{-1}\circ f\circ\alpha)$;
now apply e).
\qed\end{enumerate}\end{proof}
The above notion composition for relations is,
like that of `tightening', 
from \mycite{Section~3}{Multifunctions}. 
Mapping compact sets to compact sets is a property
which turns out useful below. It includes
both compact relations 
(Lemma~\ref{l:Implications}k)
and continuous functions:

\begin{myexample} \label{x:ClosedMap} 
\begin{enumerate}
\item[a)]
Let $f:X\to Y$ be a single-valued
continuous function.
Then $f$ maps compact sets to compact sets.
\item[b)]
The inverse $(\rhosd^d)^{-1}$
of the $d$-dimensional signed digit representation
maps compact set to compact sets.
\item[c)]
The functions $\id:x\to x$ and $\sgn:\IR\to\{-1,0,1\}$ both map compact
sets to compact sets; however their Cartesian
product $\id\times\sgn$ does not map compact
$\{(x,x):-1\leq x\leq 1\}$ to a compact set.
\end{enumerate}
\end{myexample}
Indeed, the signed digit representation
$\rhosd$ is well-known \emph{proper}
\cite[pp.209-210]{Weihrauch}, i.e. preimages
of compact sets are compact.

Focusing on complete separable metric spaces and pointwise compact multifunctions,
strong continuity is in view of Fact~\ref{f:VascoPeter94} (in general strictly) 
stronger than relative computability;
whereas weak continuity is (again in general strictly)
weaker than relative computability:

\begin{myexample} \label{x:VascoPeter94}
\begin{enumerate}
\item[a)]
The relation (\ref{e:VascoPeter94}) 
from Example~\ref{x:Examples}d)
is not computable relative to any oracle.
\item[b)]
The relation from Example~\ref{x:Examples}c) 
is (uniformly strongly continuous but, lacking
pointwise compactness) not computable relative to any oracle.
\item[c)]
The closure of the relation from Example~\ref{x:Examples}b),
that is with graph $\big([0,2/3]\times\{0\}\big)\cup\big([1/3,1]\times\{1\}\big)$,
is computable but not strongly continuous.
\end{enumerate}
\end{myexample}
\begin{proof} 
\begin{enumerate}
\item[a)] by contradiction: 
Suppose some oracle machine $\calM$ computes this relation.
On input of the rational sequence $(0,0,0,\ldots)$ as 
a $\myrho$--name of $x:=0$ it thus outputs a $\myrho$--name
of $y=0$, i.e. a rational sequence $(p_m)$ with $|p_m|<2^{-m}$.
In particular it prints $p_1>-1/2$ after having read only
finitely many elements from the input sequence; 
say, up to the $(N-1)$-st element. Now consider the behavior of
$\calM$ on the input sequence $(0,0,\ldots,0,2^{-N},2^{-N},\ldots)$
as $\myrho$--name of $x':=2^{-N}$: Its output sequence $(p'_m)$ 
will, again, begin with $p'_1=p_1>-1/2$ 
and thus cannot be a $\myrho$--name of $-1$. 
Since $g(x')=\{-1,0,2/(1+2^N)\}$, it must therefore satisfy
$|p'_m-y|<2^{-m}$ for all $m$ and for one of $y=0=:y_0$ or
$y=2/(1+2^N)=:y_1$. 
In particular, $p'_{N+1}$ satisfies
$y_j\in\ball(p'_{N+1},2^{-N-1})\not\ni y_{1-j}$
for the unique $j\in\{0,1\}$ with $y=y_j$
and is printed upon reading
only the first, say, $N'\geq N$ elements of 
$(0,0,\ldots,2^{-N},2^{-N},\ldots)$.
Finally it is easy to extend this finite sequence to a
$\myrho$--name of some $x''$ close to $x'$ with
$y_j\not\in g(x'')\ni y_{1-j}$;
and upon this input $\calM$ will now, again, 
output elements $p'_1,\ldots,p'_{N+1}$ 
which, however, cannot be extended to a $\myrho$--name of any
$y''\in g(x'')$: contradiction.
\item[b)] see \cite[p.24]{VascoPeter94}.
\item[c)] Immediate.
\qed\end{enumerate}\end{proof}
For relations with 
discrete range, on the other hand, we have

\begin{theorem} \label{t:Pauly}
Let $X$, $Y$ be computable metric spaces 
\mycite{Definition~8.1.2}{Weihrauch}.
\\ 
If $Y$ is discrete and
$f : \subseteq X \toto Y$ weakly continuous,
then $f$ is relatively computable. 
\end{theorem}
\begin{proof} 
Since $Y$ is discrete, 
$\varepsilon:=\min_{y\neq y'} d(y,y')>0$.
Now to $y\in Y$ consider the set
\[ U_y \;:=\; \big\{ x\in\dom(f): \exists\delta>0 \;
\forall x'\in\ball(x,\delta)\cap\dom(f) \; \exists y'\in f(x')\cap\ball(y,\varepsilon)\big\} \]
and note that it is open in $\dom(f)$ because
$y'\in\ball(y,\varepsilon)$ requires $y'=y$.
Hence $U_y=\dom(f)\cap\bigcup_{j\in\IN} \ball(q_{j,y},1/n_{j,y})$
for certain $n_{j,y}\in\IN$ and $q_{j,y}$ from the fixed
dense subset of $X$.
Now consider an encoding of (names of) these $q_{j,y}$ and 
$n_{j,y}$ as oracle.
Then, given $x\in\dom(f)$, search for some $(j,y)$ with
$x\in\ball(q_{j,y},1/n_{j,y})\subseteq U_y$:
when found, such $y$ by construction belongs to $f(x)$
and, conversely, weak continuity asserts $x$ to belong to
$U_y$ for some $y$.
\qed
\end{proof}

\subsection{Motivation for Uniform Continuity}
Many proofs of uncomputability of relations
or of topological lower bounds \cite{DiscreteAdvice}
apply weak continuity as a necessary condition:
\emph{merely} necessary, in view of the above example,
and thus of limited applicability.
The rest of this work thus explores topological 
conditions stronger than weak continuity yet necessary 
for relative computability. 

Uniform continuity of functions is
such a stronger notion --- and an important concept
of its own in mathematical analysis --- yet does
not straightforwardly (or at least not unanimously)
extend to multifunctions. Guided by the equivalence between
uniform continuity and relative computability for
functions with compact graph,
our aim is a topological characterization
of oracle-computable compact real relations.
One such characterization is Fact~\ref{f:VascoPeter94}; 
however we would like
to avoid (second-order) quantifying over tightenings.

To this end observe that every (relatively) computable function $f$
is (relatively) effectively locally uniformly continuous
\mycite{Theorem~6.2.7}{Weihrauch}, that is, uniformly 
continuous on every compact subset $K\subseteq\dom(f)$ \cite{Kreitz}:
\[ \forall\varepsilon>0 \; \exists\delta>0 \; \forall x\in K\; 
\forall x'\in\ball(x,\delta)\cap K: \; d\big(f(x),f(x')\big)<\varepsilon \enspace . \]
This suggests to look for related concepts for multifunctions, i.e. 
where $\delta$ does not depend on $x$. Uniform weak continuity 
in the sense of Definition~\ref{d:Main}c), however, fails to strengthen
weak continuity because it allows $y$ to depend on $\varepsilon$.

\section{Henkin-Continuity}
In view of the above discussion, we seek for an order on the four quantifiers
\[ \forall x\in\dom(f),\quad \exists y\in f(x), \quad \forall\varepsilon>0,
  \quad \exists\delta>0 \]
such that $y$ does not depend on $\varepsilon$ and $\delta$ does not depend on $x$.
This cannot be expressed in classical first-order logic and has spurred the
introduction of the non-classical so-called \textsf{Henkin Quantifier}
\cite{DependenceLogic}
\[ \QH(x,y,\varepsilon,\delta) \quad=\quad
\left(\begin{array}{cc} \forall x & \exists y \\ \forall\varepsilon & \exists\delta 
\end{array}\right) \]
where the suggestive writing indicates that very condition:
that $y$ may depend on $x$ but not on $\varepsilon$ 
while $\delta$ may depend on $\varepsilon$ but not on $x$.
We thus adopt from \cite[p.380]{Beeson}
the following\footnote{%
Its generalization from metric to uniform spaces is 
immediate but beyond our purpose.}

\begin{definition} \label{d:Henkin}
Call $f$ \textsf{Henkin-continuous} if the following holds:
\begin{equation} \label{e:Henkin}
\left(\begin{array}{cc}
\forall \varepsilon>0 & \exists \delta>0 \\[0.5ex]
\forall x\in\dom(f) \; & \exists y\in f(x)
\end{array}\right) \;\;
\forall x'\in\ball(x,\delta)\cap\dom(f) \quad
\exists y'\in\ball(y,\varepsilon)\cap f(x') \enspace . 
\end{equation}
\end{definition}
Observe that uniform strong continuity implies Henkin-continuity;
from which in turn follows both weak continuity and uniform weak continuity.
In fact, Henkin-continuity is strictly stronger than the latter two:

\begin{myexample}
\begin{enumerate}
\item[a)]
The relation $g$ from Examples~\ref{x:Examples}d) and \ref{x:VascoPeter94}a) is 
(compact and both weakly continuous and uniformly weakly continuous but)
not Henkin-continuous. \medskip
\item[b)] It does, however, satisfy 
\quad $\displaystyle 
\binom{\forall\varepsilon>0 \;\;\exists\delta>0}{\forall x,x' \; \exists y\in g(x)}
\;\exists y'\in g(x') \; \Big(x'\in\ball(x,\delta)\rightarrow y'\in\ball(y,\varepsilon)\Big)$.
\medskip
\item[c)]
The relations from Examples~\ref{x:Examples}b) and \ref{x:VascoPeter94}c)
are (computable and) Henkin-continuous.
\end{enumerate}
\end{myexample}
\begin{proof} 
\begin{enumerate}
\item[a)] by contradiction:
Suppose $y=y(x)$ satisfies Equation~(\ref{e:Henkin}).
Now let $\varepsilon:=1/2$ and consider $\delta:=\delta(\varepsilon)$
according to Equation~(\ref{e:Henkin}).
Then $y(x)=-1$ is impossible for all $0<x<\delta$,
as $x':=(x-\delta)/2<0$ implies $g(x')=\{0\}$
which is disjoint to $\ball(y,\varepsilon)$.
Now consider $\varepsilon':=\delta\cdot2/3$ and $\delta':=\delta(\varepsilon')$.
We claim that $y(x)=-1$ is necessary for all $x>\varepsilon'$,
this leading to a contradiction for $\delta\cdot2/3<x<\delta$.
Indeed, in case $y(x)=x$, rational 
$x'\in\ball\big(x,\min\{\delta',\delta/3\}\big)$
implies $g(x')=\{0\}$ which is disjoint to
$\ball(y,\varepsilon')$;
whereas in case $y(x)=0$, irrational 
$x'\in\ball\big(x,\min\{\delta',\delta/3\}\big)$
implies $g(x')=\{x'\}$ which is disjoint to
$\ball(y,\varepsilon')$.
\item[b)]
Let $\delta:=\varepsilon$ and take
$y:=-1$ in case $x,x'>0$; $y:=0$ in case $x\leq 0$;
and $\{y\}:=g(x)\cap[0,1]$ in case $x'\leq 0<x$.
\item[c)]
For $x\leq\tfrac{1}{2}$ choose $y:=0$
and for $x>\tfrac{1}{2}$ choose $y:=1$;
independently, choose $\delta:=\tfrac{1}{6}$.
\qed\end{enumerate}\end{proof}

\subsection{Further Examples and Some Properties}
Recall that, for single-valued functions, 
Henkin-continuity coincides with uniform continuity. 

\begin{myexample} \label{x:Representations}
Recall from the Type-2 Theory of Effectivity (TTE) 
the Cauchy representation $\myrhoC$ 
\mycite{Definition~4.1.5}{Weihrauch}
and the signed digit representation $\rhosd$ 
\mycite{Definition~7.1.4}{Weihrauch} of real numbers.
\begin{enumerate}
\item[a)]
$\rhosd:\subseteq\Cantor\to\IR$ is not uniformly continuous
\item[b)]
nor is the restriction 
$\myrhoC|^{[0,1]}:\subseteq\Cantor\to[0,1]$;
cmp. \mycite{Example~7.2.3}{Weihrauch}.
\item[c)]
However for every compact $K\subseteq\IR$, the restriction 
$\rhosd|^{K}:\subseteq\Cantor\to K$ is uniformly
(i.e. Henkin-) continuous;
\item[d)]
and so are the restrictions $\myrhoC|_{_C}:C\to\IR$ 
and $\rhosd|_{_C}:C\to\IR$ for any compact $C\subseteq\Cantor$.
\item[e)]
$\displaystyle \myrhoC^{-1}:\IR\toto\Cantor, \quad
 \IR\ni x\mapsto \{\bar\sigma:\myrhoC(\bar\sigma)=x\}$,
the inverse of the Cauchy representation, is Henkin-continuous.
\item[f)]
Let $\langle\,\cdot\,,\,\cdot\,\rangle:\IN\times\IN\to\IN$ be an integer
pairing function with $\langle n,m\rangle\geq n+m$ for every $n,m\in\IN$.
Then the string pairing function $\{0,1\}^{\omega\times\omega}\to\Cantor$,
$(b_{\langle n,m\rangle})_{_{n,m\in\omega}}\mapsto (b_k)_{_{k\in\omega}}$
is 1-Lipschitz (and thus uniformly) continuous.
\end{enumerate}
\end{myexample}
\begin{proof}
\begin{enumerate}
\item[a)]
Consider some large integer $x=2^k\in\IN$ with $\rhosd$--name 
\texttt{10$\cdots$0.0$\cdots$}
(each digit \sdzero, \sdone, \sdminus, 
 and the point \sddot encoded as a constant-length
 string over $\{0,1\}^*$).
Then modifying this name $\bar\sigma$ at the $k$-th position 
affects the value $\rhosd(\bar\sigma)$ by an absolute value
of 1. In particular, to $\varepsilon:=1$, 
$\delta>0$ satisfying 
\[ d(\bar\sigma,\bar\tau)<\delta \quad\Rightarrow\quad
  d\big(\rhosd(\bar\sigma),\rhosd(\bar\tau)\big)<\varepsilon \]
must depend on the value of $x=2^k$, i.e. on $\bar\sigma$.
\item[b)]
Fix $k\in\IN$, 
and consider integers $a_n:=2^{k+n}$ and $b_n:=3\cdot 2^{k+n}$.
Hence the concatenation $\bar\sigma$ of binary-encoded numerators $a_n$ 
and denominators $b_n$ constitutes a $\myrhoC$--name of $x:=1/3$. 
Note that the secondmost-significant digit of $b_1$ resides 
roughly at position $\#k$ in $\bar\sigma$. Hence switching
to $a_n':=a_n$ and $b_n':=2\cdot 2^{k+n}$ 
yields $\bar\sigma'$ of metric
distance to $\bar\sigma$ of order $\delta=2^{-k}$; 
whereas the value 
$x'=\myrhoC(\bar\sigma')=1/2$ changes by $\varepsilon=1/6$.
\item[c)]
First consider the case $K=[0,1]$.
Then, modifying the $k$-th digit $b_k\in\{0,+1,-1\}$ 
of a signed digit expansion $\sum_{n=0}^{\infty} b_n 2^{-n}$
affects its value by no more than $2^{-k}$.
In the general case, let $2^{\ell}$ denote a bound on
$K$. Then, similarly, modifying the $k$-th position
of a signed digit expansion $\sum_{n=-N}^\infty b_n 2^{-n}$
affects its value by no more than $2^{\ell-k}$.
\item[d)]
Like any admissible representation, $\myrhoC$ and $\rhosd$ are continuous;
hence uniformly continuous on compact subsets.
\item[e)]
To $\varepsilon=2^{-k}>0$ let $\delta:=2^{-k}$.
Now consider arbitrary $x\in\IR$
and as $\myrhoC$--name $\bar\sigma$ 
the (binary encodings of numerators and denominators of the) 
dyadic sequence $q_n:=\lfloor x\cdot 2^{n+1}\rfloor/2^{n+1}$.
In fact it holds $|x-q_n|\leq2^{-n-1}\leq 2^{-n}$.
Now $x'\in\ball(x',\delta)$ has 
$|x'-q_n|\leq2^{-k}+2^{-n-1}\leq 2^{-n}$ 
for $n\leq k-1$. Therefore the first $k-1$ 
elements of $(q_n)$, and in particular the first
$k-1$ symbols of $\bar\sigma$, extend to a 
$\myrhoC$--name $\bar\tau$ of $x'$;
i.e. such that $d(\bar\sigma,\bar\tau)<\varepsilon$.
\item[f)]
Modifying the the argument at index $(n,m)$
affects the image at index $\langle n,m\rangle\geq n+m$,
i.e. the metric at weight $\leq2^{-(n+m)}$.
\qed\end{enumerate}\end{proof}
A classical property both of continuity and uniform continuity
is closure under restriction and under composition. 
Also Henkin-continuity passes these 
(appropriately generalized) sanity checks:

\begin{observation} \label{o:Composition}
\begin{enumerate}
\item[a)]
Let $f:\subseteq X\times Y$ be Henkin-continuous and 
tighten $g:\subseteq X\times Y$.
Then $g$ is Henkin-continuous, too.
\item[b)]
If $f:\subseteq X\times Y$ and $g:\subseteq Y\times Z$ are
Henkin-continuous, then so is $g\circ f\subseteq X\times Z$.
\end{enumerate}
\end{observation}
\begin{proof}
\begin{enumerate}
\item[a)]
For $g$ loosening $f$
and in the definition of Henkin-continuity of $g$,
the universal quantifiers range over a subset,
and the existential quantifiers range over a superset,
of those in the definition of Henkin-continuity of $f$.
\item[b)]
By hypothesis, we have
\begin{gather}
\left(\begin{array}{cc}
\forall \varepsilon>0 & \exists \delta>0 \\[0.5ex]
\forall y\in\dom(g) \; & \exists z\in g(y)
\end{array}\right) \;\; 
\forall y'\in\ball(y,\delta)\cap\dom(g) \quad
\exists z'\in\ball(z,\varepsilon)\cap g(y') 
\label{e:temp1}
\\[1ex]
\left(\begin{array}{cc}
\forall \delta>0 & \exists \gamma>0 \\[0.5ex]
\forall x\in\dom(f) \; & \exists y\in f(x)
\end{array}\right) \;\; 
\forall x'\in\ball(x,\gamma)\cap\dom(f) \quad
\exists y'\in\ball(y,\delta)\cap f(x')
\label{e:temp2}
\end{gather}
Thus, to $\varepsilon>0$, take $\delta>0$ 
according to Equation~(\ref{e:temp1})
and in turn $\gamma>0$ 
according to Equation~(\ref{e:temp2}).
Similarly, to 
$x\in\dom(g\circ f)\subseteq\dom(f)$,
take $y\in f(x)\subseteq\dom(g)$ 
according to Equations~(\ref{e:temp2})
and (\ref{e:Composition});
and in turn $z\in g(y)$
according to Equation~(\ref{e:temp1}).
This $z$ thus belongs to $\big(g\circ f\big)(x)$ 
and was obtained independently of $\varepsilon$,
nor does $\gamma$ depend on $x$.
Moreover
to $x'\in\ball(x,\gamma)\cap\dom(g\circ f)$
there is a $y'\in\ball(y,\delta)\cap f(x')
\subseteq\ball(y,\delta)\cap\dom(g)$;
to which in turn there is a 
$z'\in\ball(z,\varepsilon)\cap g(y')$,
i.e. $z'\in\ball(z,\varepsilon)\cap \big(g\circ f\big)(x')$.
\qed\end{enumerate}\end{proof}
The following further example in Item~b) turns out as rather useful:

\begin{proposition} \label{p:SignedDigit}
\begin{enumerate}
\item[a)]
Every $x\in\IR$ has a signed digit expansion
\begin{equation} \label{e:SignedDigit}
x\;\;=\;\;\sum\nolimits_{n=-N}^\infty a_n 2^{-n},
\qquad a_n\in\{\sdzero,\sdone,\sdminus\}
\end{equation}
with \emph{no} consecutive digit pair $\sdone\sdone$ nor $\sdminus\sdminus$
nor $\sdone\sdminus$ nor $\sdminus\sdone$. 
\item[b)]
For $k\in\IN$,
each $|x|\leq\tfrac{2}{3}\cdot2^{-k}$ 
admits such an expansion with $a_n=0$ for all $n\leq k$.
And, conversely, 
$x=\sum_{n=k+1}^\infty a_n 2^{-n}$ with 
$(a_n,a_{n+1})\in\{\sdone\sdzero,\sdminus\sdzero,\sdzero\sdone,\sdzero\sdminus,\sdzero\sdzero\}$
for every $n$ requires $|x|\leq\tfrac{2}{3}\cdot2^{-k}$.
\item[c)]
Let $x=\sum_{n=-N}^\infty a_n 2^{-n}$ be a signed digit expansion
and $k\in\IN$ such that 
$(a_n,a_{n+1})\in\{\sdone\sdzero,\sdminus\sdzero,\sdzero\sdone,\sdzero\sdminus,\sdzero\sdzero\}$
for each $n>k$.
Then every $x'\in[x-2^{-k}/3,x+2^{-k}/3]$ admits a signed digit expansions
$x'=\sum_{n=-N}^\infty b_n 2^{-n}$ with $a_n=b_n\forall n\leq k$.
\item[d)]
Let $\Sigma:=\{\sdzero,\sdone,\sdminus,\sddot\}$.
The inverse $\rhosd^{-1}:\IR\toto\Sigma^\omega$ 
of the signed digit representation
is Henkin-continuous. 
\end{enumerate}
\end{proposition}
\begin{proof}
\begin{enumerate}
\item[a)]
Start with an arbitrary signed digit expansion $(a_n)$ of $x$
and replace, starting from the most significant digits,
\begin{enumerate} \itemsep0pt%
\item[i)] any occurrence of $\sdzero\sdone\sdone$ with $\sdone\sdzero\sdminus$,
\item[ii)] any occurrence of $\sdzero\sdminus\sdminus$ with $\sdminus\sdzero\sdone$,
\item[iii)] any occurrence of $\sdzero\sdone\sdminus$ with $\sdzero\sdzero\sdone$,
\item[iv)] any occurrence of $\sdzero\sdminus\sdone$ with $\sdzero\sdzero\sdminus$.
\end{enumerate}
Note that these substitutions do not affect the value $\sum_{n=-N}^\infty a_n2^{-n}$.
Moreover the above four cases are the only possible involving one of
$\sdone\sdone$ or $\sdminus\sdminus$ or $\sdone\sdminus$ or $\sdminus\sdone$ 
because, by induction hypothesis and proceeding from left (most significant) to right,
no such combination was left \emph{before} of the current position.
On the other hand, rewriting Rule~i) may well introduce a new occurrence of
$\sdone\sdone$ before the current position; this is illustrated in the example
of $\sdzero\sdone\sdzero\sdone\sdzero\sdone\sdone$. Similarly 
for $\sdminus\sdminus$ in Rule~ii). 
Therefore, we apply the rules in two loops: 
\begin{itemize}\itemsep0pt
\item An infinite outer one for $n=-N,\ldots,0,1,2,\ldots$, \\
maintaining that neither $\sdone\sdone$ nor $\sdminus\sdminus$ nor
$\sdone\sdminus$ nor $\sdminus\sdone$ occurs before position $n$
\item one application of rules i) to iv) to remove a
possible occurrence at position $n$
\item followed by a finite inner loop for $j$ running from $n$ back to $-N$,
iteratively removing occurrences which may have been newly introduced
at position $j$.
\end{itemize}
Observe that, after each termination of the inner loop,
no occurrence remains before or at position $n$.
Hence the process converges and yields an equivalent signed digit
expansion with the desired property.
\item[b)]
Shifting/scaling reduces to the case $k=0$;
and negation to the case $x>0$.
\\
$\tfrac{2}{3}=\sdzero.\sdone\sdzero\sdone\sdzero\ldots$ is 
an expansion with the claimed properties. 
So turn to $0<x<\tfrac{2}{3}$ and, indirectly, w.l.o.g. suppose $a_0=\sdone$. 
Extend this to a signed digit expansion of least value $\sum_{n=0}^\infty a_n2^{-n}=x$
with no consecutive $\sdone\sdone,\sdminus\sdminus,\sdone\sdminus,\sdminus\sdone$.
Due to monotonicity,
this is attained by including digit $\sdminus$ whenever admissible, namely
$\sdone.\sdzero\sdminus\sdzero\sdminus\ldots$
of value $x=\tfrac{2}{3}$: a contradiction.
\\
For the converse, similarly observe that $\sdzero.\sdone\sdzero\sdone\sdzero\ldots$
has the largest value among all signed digit expansions with the claimed properties;
and its value is $\tfrac{2}{3}$.
\item[c)]
Let $x'':=\sum_{n=-N}^k a_n 2^{-n}$ and observe
that $x-x''=\sum_{n=k+1}^\infty a_n 2^{-n}$ is by hypothesis
a signed digit expansion satisfying 
$(a_n,a_{n+1})\in\{\sdone\sdzero,\sdminus\sdzero,\sdzero\sdone,\sdzero\sdminus,\sdzero\sdzero\}$
for all $n\geq k+1$, hence $0\leq x-x''\leq\tfrac{2}{3}\cdot2^{-k}$ by b).
In addition with the hypothesis $|x-x'|\leq2^{-k}/3$,
we conclude that $x'-x''=(x'-x)+(x-x'')\in[-\tfrac{1}{3}\cdot2^{-k},2^{-k}]$ 
admits a signed digit expansion (possibly using combinations like $\sdone\sdone$)
$x'-x''=\sum_{n=k+1}^\infty b_n2^{-n}$.
Thus $x'=(x'-x'')+x''=\sum_{n=-N}^k a_n 2^{-n}+\sum_{n=k+1}^\infty b_n2^{-n}$
is an expansion with the claimed properties.
\item[d)]
To $2^{-k}\geq\varepsilon>0$ let $\delta:=\tfrac{2}{3}\varepsilon$.
To $x\in\IR$ let $\bar\sigma$ be a $\rhosd$--name $\bar\sigma$ 
\mycite{Definition~7.2.4}{Weihrauch} encoding the signed digit expansion
$(a_n)$ of $x$ according to a).
Due to c), every 
$x'\in\cball(x,\delta)\subseteq\cball(x,\cdot2^{-(k-1)}/3)$ 
admits a signed digit expansion $(b_n)$ coinciding with 
$(a_n)$ for all $n\leq k-1$. 
Since every $\rhosd$--name includes the binary separator symbol, 
an appropriate name $\bar\sigma'$ encoding $(b_n)$ agrees with $\bar\sigma$
for at least the first $k+1$ symbols, i.e. has distance at most
$2^{-k}\leq\varepsilon$.
\qed\end{enumerate}\end{proof}

\subsection{Other Characterizations and Tools}
Let us call a mapping $\lambda:\IN\to\IN$ a \textsf{modulus};
and say that a multifunction $f:\subseteq X\toto Y$ 
is \textsf{$\lambda$-continuous in} $(x,y)\in f$ if,
to every $m\in\IN$ and every $x'\in\dom(f)\cap\cball(x,2^{-\lambda(m)})$
there exists some $y'\in f(x')\cap\cball(y,2^{-m})$.
Here, $\cball(x,r):=\{x'\in X:d(x,x')\leq r\}$ denotes
the closed ball of radius $r$ around $x$.
Now Skolemization of ``$\forall\varepsilon>0\exists\delta>0$'' yields

\begin{observation} \label{o:Modulus}
A multifunction $f:\subseteq X\toto Y$ is Henkin-continuous
~iff~ there exists a modulus $\lambda$ such that,
for every $x\in\dom(f)$, there exists $y\in f(x)$
such that $f$ is $\lambda$-continuous in $(x,y)$; \\
equivalently: if, for every $x\in\dom(f)$, 
$f$ admits some single-valued total 
selection $f_x:X\to Y$ $\lambda$-continuous in $\big(x,f_x(x)\big)$
(but possibly not continuous anywhere else, see Example~\ref{x:noLocSel} below).
\end{observation}

\begin{definition}
\begin{enumerate}
\item[a)] For $L>0$, a multifunction $f:\subseteq X\toto Y$ is
\textsf{$L$-Lipschitz} if
\begin{equation} \label{e:Lipschitz}
\forall x\in\dom(f) \;\;
\exists y\in f(x) \;\;
\forall x'\in\dom(f) \;\;
\exists y'\in\ball\big(y,L\cdot d(x,x')\big)\cap f(x') \enspace .
\end{equation}
\item[b)]
Call a family $f_i:\subseteq X_i\toto Y_i$ ($i\in I$) of multifunctions
\textsf{equicontinuous} if they share a common modulus
in the sense that the following holds:
\begin{equation} \label{e:Equi}
\binom{
\forall \varepsilon>0 \;\; \exists \delta>0}{
\forall i\in I 
\;\forall x\in\dom(f_i) \; \exists y\in f_i(x)} \;\;
\forall x'\in\ball(x,\delta) \quad
\exists y'\in\ball(y,\varepsilon)\cap f(x') \enspace .
\end{equation}
\end{enumerate}
\end{definition}
So every Lipschitz relation is Henkin-continuous;
and every family of total $L$-Lipschitz relations
is equicontinuous. 
The proof of Proposition~\ref{p:SignedDigit}d)
reveals Item~a) of the following

\begin{myexample} \label{x:noLocSel}
\begin{enumerate}
\item[a)]
For $\Sigma=\{\sdzero,\sdone,\sdminus,\sddot\}$,
the inverse $\rhosd^{-1}:\IR\toto\Sigma^\omega$ 
of the signed digit representation\footnote{Note that
proceeding from alphabet $\Sigma$ to $\{0,1\}^2$ 
affects the Lipschitz constant by a factor of 2.},
is $\tfrac{3}{2}$-Lipschitz.
\item[b)]
The relation 
\[ f \quad :=\quad \big\{(0,0)\big\}\;\cup\; \bigcup\nolimits_{k\in\IN}
  \big[2^{-k},\max\{1,3\cdot2^{-k}\}\big]\times\big\{2^{-k}\big\} 
  \quad\subseteq\quad [0,1]\times[0,1] \enspace . \]
depicted in Figure~\ref{f:noLocSel} is compact and 1-Lipschitz.
Moreover, $f$ is computable but has no locally continuous selection in $x_0=0$.
\end{enumerate}
\end{myexample}
\begin{figure}
\centerline{\includegraphics[width=0.7\textwidth]{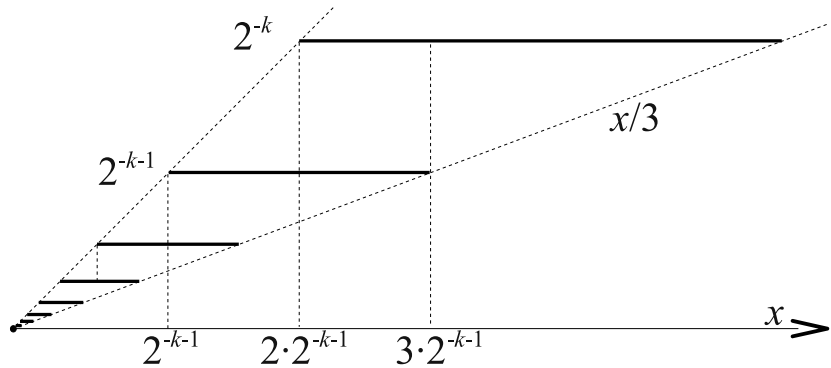}}
\caption{\label{f:noLocSel}Computable compact relation
with no locally continuous selection in $x_0=0$.}
\end{figure}
Concerning Example~\ref{x:noLocSel}b), 
the ratio $\min\{|y-y'|:y\in f(x),y'\in f(x')\}|/|x-x'|$ becomes 
worst for $x=2\cdot 2^{-k-1}-\varepsilon$ (hence $f(x)=\{2^{-k-1}\}$, i.e.
$y=2^{-k-1}$) and $x'=3\cdot2^{-k-1}+\varepsilon$
(hence $f(x')=\{2^{-k}\}$, i.e. $y=2^{-k}$).
Moreover every $(x,y)\in f$ satisfies $x/3\leq y\leq x$.
Thus the following algorithm computes $f$:
Given $x\in[0,1]$ in form of a nested sequence $[a_n,b_n]$
of intervals with rational endpoints $b_n-a_n\leq2^{-n-1}$,
test whether $[a_n,b_n]\subseteq [2^{-n},3\cdot2^{-n}]$ holds:
if not, output $[a_n/3,b_n]$ and proceed to interval $\#n+1$, 
otherwise switch to outputting the constant sequence $[2^{-n},2^{-n}]$.
Note that for $x=0$, the output sequence $[a_n/3,b_n]$ will indeed
converge to $y=0$. In case $3\cdot2^{-k-1}<x\leq2\cdot2^{-k}$ on
the other hand, $[a_k,b_k]\subseteq [2^{-k},3\cdot2^{-k}]$ holds
and will result in the output of $y=2^{-k}\in f(x)$, compliant 
with possible previous intervals $[a_n/3,b_n]\supseteq [x/3,x]\supseteq f(x)$.
In the final case $2\cdot2^{-k-1}<x\leq3\cdot2^{-k-1}$,
at least one of $[a_k,b_k]\subseteq [2^{-k},3\cdot2^{-k}]$ 
and $[a_k,b_k]\subseteq [2^{-k-1},3\cdot2^{-k-1}]$ holds;
hence the algorithm will produce $2^{-n}$ either for 
$n=k$ or for $n=k+1$.\qed

\begin{proposition} \label{p:Arno}
\begin{enumerate}
\item[a)]
$I$ denote an ordinal and 
$f_i:\subseteq X\toto Y$ ($i\in I$) an equicontinuous 
family of pointwise compact multifunctions and decreasing
in the sense that $f_j$ tightens $f_i$  whenever $j>i$.
Then $f(x):=\bigcap_{i:f_i(x)\neq\emptyset} f_i(x)$ 
is again pointwise compact and Henkin-continuous
a tightening of each $f_i$. \\
Moreover, if all $f_i$ are $\lambda$-continuous,
then so is $f$.
\item[b)]
Let $f:X\toto Y$ be $\lambda$-continuous
and pointwise compact for some modulus $\lambda$. 
Then $f$ has a 
minimal $\lambda$-continuous pointwise compact tightening.
\end{enumerate}
\end{proposition}
\begin{proof}
\begin{enumerate}
\item[a)]
Since the case of a finite $I$ is trivial,
it suffices to treat the case $I=\IN$ of a sequence;
the general case then follows by transfinite induction.
Let $x\in\dom(f_i)$. Then $f_j(x)\subseteq f_i(x)$ for
each $j>i$, and hence $f(x)=\bigcap_{j\geq i} f_j(x)\subseteq f_i(x)$ 
is (compact and) the intersection of non-empty compact 
decreasing sets: $f(x)\neq\emptyset$, $x\in\dom(f)$.
Moreover let $\varepsilon>0$ be arbitrary
and consider an appropriate $\delta$
according to Equation~(\ref{e:Equi}) 
independent of $x$; similarly take $y_j\in f_j(x)$ 
independent of $\varepsilon$ as asserted by equicontinuity.
Then the sequence $(y_j)_{j>i}$ belongs to compact $f_j(x)$
and thus has some accumulation point $y\in f_j(x)\subseteq f_i(x)$ 
for each $j$: thus yields $y\in f(x)$ independent of $\varepsilon$.
W.l.o.g $y_j\to y$ by proceeding to a subsequence.
Now let $d(x,x')\leq\delta$. Then
by hypothesis there exists $y_j'\in f_j(x')$ with
$d(y_j,y_j')\leq\varepsilon$;
and, again, an appropriate 
subsequence of $(y_j')$ converges to some $y'\in f(x')$.
Moreover, $d(y,y')\leq d(y,y_j)+d(y_j,y_j')+d(y_j',y')
\leq d(y,y_j)+\varepsilon+d(y',y_j')\to\varepsilon$.
\item[b)]
Consider the family $\calF$ of all 
$\lambda$-continuous and pointwise compact tightenings of $f$.
According to a), these form a
\emph{directed complete partial order} (\textsf{dcpo}) 
with respect to total restriction. More explicitly,
apply Zorn's Lemma to get a maximal chain $(f_i)$, $i\in I$.
Then a) asserts that $g(x):=\bigcap_{i:f_i(x)\neq\emptyset} f_i(x)$ 
defines a $\lambda$-continuous and pointwise compact tightening
of $f$. In fact a minimal one: 
If $h\in\calF$ tightens $g$, 
then $h=f_j$ for some $j\in I$ because of the maximality
of $(f_i)_{_{i\in I}}$; hence $g$ tightens $f_j$.
\qed\end{enumerate}\end{proof} 

\subsection{Relative Computability requires Henkin-Continuity}
With the above examples and tools,
it is now easy to establish

\begin{theorem} \label{t:Henkin}
Let $K\subseteq\IR$ be compact.
\begin{enumerate}
\item[a)]
If $f:K\toto\IR$ is computable relative to some oracle,
then it is Henkin-continuous.
\item[b)]
More precisely suppose $F:\subseteq\Cantor\toto\Cantor$
is a Henkin-continuous 
$(\rhosd,\rhosd)$--\textsf{multi}realizer of $f:K\toto\IR$
(recall Lemma~\ref{l:Tighten})
which maps compact sets to compact sets.
Then $f$ itself must be Henkin-continuous, too;
and has a Henkin-continuous tightening 
$g:K\toto\IR$ mapping compact sets to compact sets.
\item[c)]
Conversely, if $f:K\toto\IR$ is Henkin-continuous
and maps compact sets to compact sets,
then $F:=\rhosd^{-1}\circ f\circ\rhosd|^K$
is a Henkin-continuous 
$(\rhosd,\rhosd)$--multirealizer of $f$
which maps compact sets to compact sets.
\end{enumerate}
\end{theorem}
\begin{proof}
\begin{enumerate}
\item[a)]
Recall \mycite{Section~3}{Weihrauch} that a real relation
is relatively computable ~iff~ it has a continuous
$(\myrho,\myrho)$--realizer; equivalently
\mycite{Theorem~7.2.5.1}{Weihrauch}:
a continuous $(\rhosd,\rhosd)$--realizer $F$.
In particular, single-valued $F$
maps compact sets to compact sets.
Moreover, $F$ is a $(\rhosd,\rhosd)$--multirealizer 
according to Lemma~\ref{l:Tighten}f);
and has $\dom(F)=\dom(\rhosd|^K)$ compact
\cite[pp.209-210]{Weihrauch}, hence is even uniformly
continuous, i.e. Henkin-continuous.
Now apply b).
\item[b)]
Proposition~\ref{p:SignedDigit}d) asserts
$\rhosd^{-1}$ to be Henkin-continuous;
and so is 
$(\rhosd|^K)^{-1}=(\rhosd^{-1})|_K$, cmp. Observation~\ref{o:Composition}a).
Now $\range\big((\rhosd|^K)^{-1}\big)=\rhosd^{-1}[K]$ is compact;
which $F$ maps by hypothesis to some compact set $C\subseteq\Cantor$.
Therefore $\rhosd|_{_C}$ is uniformly (i.e. Henkin-)
continuous (Example~\ref{x:Representations}d);
and so is $\rhosd|_{_C}\circ F\circ(\rhosd|^K)^{-1}$
(Observation~\ref{o:Composition}b);
which, because of $C=\range\big(F\circ(\rhosd|^K)^{-1}\big)$,
coincides with $g:=\rhosd\circ F\circ\rhosd^{-1}$.
Now this $g$ by hypothesis tightens $f$;
hence $f$ is also Henkin-continuous 
(Observation~\ref{o:Composition}a).
Moreover, $g$ maps compact sets to compact sets 
according to Lemma~\ref{l:Tighten}d)
because each subterm
$\rhosd^{-1}$ \cite[pp.209-210]{Weihrauch},
$F$ (hypothesis), and 
$\rhosd$ (continuous) does so.
\item[c)]
Again, $\rhosd|^K$ and $\rhosd^{-1}$ are Henkin-continuous  
by Example~\ref{x:Representations}c) and
Proposition~\ref{p:SignedDigit}d);
hence so is the composition $F$ (Observation~\ref{o:Composition}a).
$F$ maps compact sets to compact sets according
to Lemma~\ref{l:Tighten}d);
note that $\range(f)\subseteq\IR=\dom(\rhosd^{-1})$
and $\range(\rhosd|^K)=K=\dom(f)$.
Finally, Lemma~\ref{l:Tighten}a+b) shows
$f$ to tighten $\rhosd\circ F\circ\rhosd^{-1}$.
\qed\end{enumerate}\end{proof}

\subsection[\ldots but Henkin-Continuity does not imply Relative Computability]{%
Henkin-Continuity does not imply Relative Computability}

The relation from Example~\ref{x:Examples}c)
is Henkin-continuous but not relatively computable.
On the other hand, it violates the natural condition 
of (pointwise) compactness. Instead, we modify 
Example~\ref{x:noLocSel} to obtain (counter-)

\begin{myexample} \label{x:Counter}
Let 
\begin{alignat*}{2}
f_+ \quad&:=\quad \big((-\infty,0]\times\{0\}\big)
\;\cup\;\big\{ \big(x,(-1)^n/(n+1)\big) : n\in\IN, 1/(n+1)\leq x\leq 1/n \big\} \\
f_- \quad&:=\quad \big([0,\infty)\times\{1\}\big)
\cup\big\{ \big(-x,1+(-1)^n/(n+1)\big) : n\in\IN, 1/(n+1)\leq x\leq 1/n \big\} 
\end{alignat*}
Then $f_1:=f_+\cup f_-:[-1,+1]\toto[-1,+2]$ is compact, total,
and 1-Lipschitz (hence Henkin-continuous), 
but not relatively computable; see Figure~\ref{f:Counter}.
\end{myexample}

\begin{figure}[htb]
\centerline{\includegraphics[width=0.9\textwidth,height=0.3\textheight]{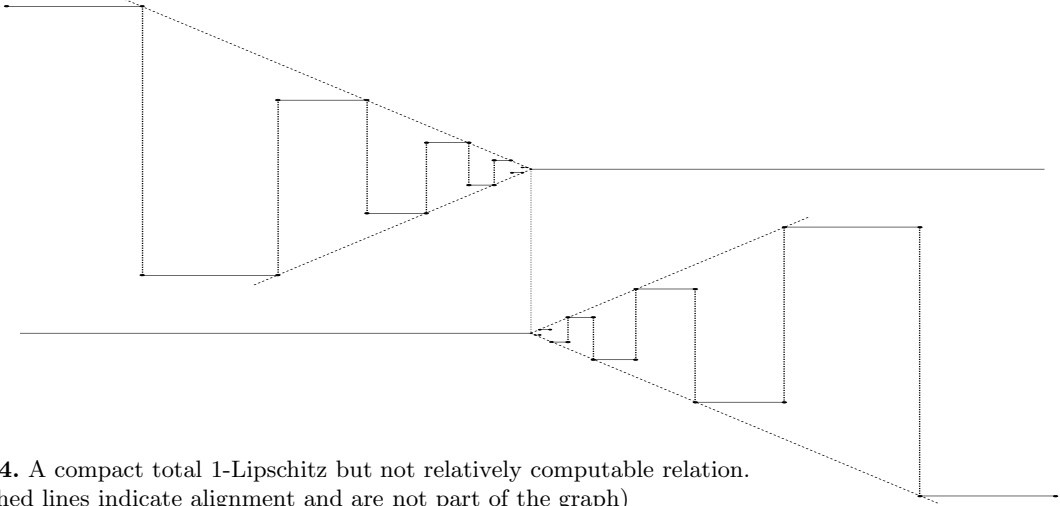}}
\vspace*{-8ex}\caption{\label{f:Counter}%
A compact total 1-Lipschitz but not relatively computable relation. \newline
(Dashed lines indicate alignment and are not part of the graph)}
\end{figure}

\begin{proof}
Both $f_+$ and $f_-$ are closed and bounded and total.
Moreover, the restriction $f_+\big|_{_{[-1,0]}}$ is 1-Lipschitz:
To $x\leq0$ set $y:=0$ and $\delta:=\varepsilon$ (1-Lipschitz);
now if $x'\leq 0$, $y':=0$ will do; and if $0<x'<\delta$, 
consider $n\in\IN$ with $1/(n+1)\leq x'\leq1/n$,
$y':=(-1)^n/(n+1)\in f_+(x')$ 
has $|y'-y|=1/(n+1)\leq x'<\delta=\varepsilon$.
Similarly, $f_-\big|_{_{[0,1]}}$ is 1-Lipschitz;
hence $f_1$ is 1-Lipschitz---but not relatively computable:
Given a name of $x=0$, the putative realizer has the choice of
producing either a name of $y_+=0$ or of $y_-=1$: knowing 
$x$ only up to some $\delta=1/n$, $n\in\IN$. In the first case,
i.e. already tied to $f_+$, switch to an input $x':=1/(n+1)$: 
clearly a point of discontinuity of $f_+$.
A similar contradiction arises in the second case.
\qed\end{proof}

\section{Iterated Henkin-Continuity}
(Counter-)Example~\ref{x:Counter} suggests to strengthen Definition~\ref{d:Henkin}:

\begin{definition} \label{d:Henkin2}
Call a total\footnote{This requirement is employed only for notational
convenience and can always be satisfied by proceeding to the
restriction $f|_{\dom(f)}$.} multifunction $f:X\toto Y$ 
\textsf{doubly Henkin-continuous} iff the following holds:
\[ \left(\begin{array}{cc}
\forall \varepsilon>0 & \exists \delta>0 \\[0.5ex]
\forall x\!\in\! X \; & \exists y\!\in\! f(x)
\end{array}\right) \;
\left(\begin{array}{cc}
\forall \varepsilon'>0 & \exists \delta'>0 \\[0.5ex]
\forall x'\!\in\!\ball(x,\delta) \; & 
\exists y'\!\in\! f(x')\cap\ball(y,\varepsilon)
\end{array}\right) \;
\forall x''\!\in\!\ball(x',\delta') \;  
\exists y''\!\in\! f(x'')\cap\ball(y',\varepsilon')
\]
Even more generally, \textsf{$\ell$-fold Henkin-continuity} 
($\ell\in\IN$) is to mean
\begin{multline} \label{e:Henkin2}
\left(\begin{array}{cc}
\forall \varepsilon_1>0 & \exists \delta_1>0 \\[0.5ex]
\forall x_1\in X \; & \exists y_1\in f(x_1)
\end{array}\right) \quad
\left(\begin{array}{cc}
\forall \varepsilon_2>0 & \exists \delta_2>0 \\[0.5ex]
\forall x_2\in\ball(x_1,\delta_1) \; & \exists y_2\in f(x_2)\cap\ball(y_1,\varepsilon_1)
\end{array}\right) \;\;
\cdots \\[1ex]
\cdots\;\; \left(\begin{array}{cc}
\forall \varepsilon_\ell>0 & \exists \delta_\ell>0 \\[0.5ex]
\forall x_\ell\in\ball(x_{\ell-1},\delta_{\ell-1}) \; & 
\exists y_\ell\in f(x_\ell)\cap\ball(y_{\ell-1},\varepsilon_{\ell-1})
\end{array}\right) \\[1ex]
\forall x_{\ell+1}\in\ball(x_\ell,\delta_\ell) \quad
\exists y_{\ell+1}\in\ball(y_\ell,\varepsilon_\ell)\cap f(x_{\ell+1}) \enspace . 
\end{multline}
\end{definition}
Generalizing Example~\ref{x:Counter}, we observe that
this notion indeed gives rise to a proper hierarchy:

\begin{myexample}[Hierarchy] \label{x:Hierarchy}
To every $\ell\in\IN$
there exists a compact total relation $f_\ell:[-1,1]\toto[-1,2]$
which is $\ell$-fold Henkin-continuous 
but not $(\ell+1)$-fold Henkin-continuous.
\end{myexample}
To this end, consider $\ell=1$ and 
recall that the relation in Figure~\ref{f:Counter}
is (1-fold) Henkin-continuous.
To $x=0$ w.l.o.g. suppose $y=0$ is chosen
and to $\varepsilon:=1/4$ some $\delta>0$.
Now consider $x':=1/n<\delta$:
Since $f_+$ is discontinuous at $x'$,
both choices $y'=s(-1)^n/(n+1)$ and $y'=-(-1)^n/(n+2)$ 
from $f(x')$ contradict 2-fold Henkin-continuity
for some $x''=x'\pm\varepsilon'$.
\\
Figure~\ref{f:Counter2} depicts an iteration $f_2$ 
of Figure~\ref{f:Counter} which, similarly, can be
seen 2-fold Henkin-continuous but not 3-fold.
Repeating this iteration, one obtains a fractal sequence
$f_\ell$ with the claimed properties.

\begin{figure}[htb]
\includegraphics[width=0.98\textwidth,height=0.35\textheight]{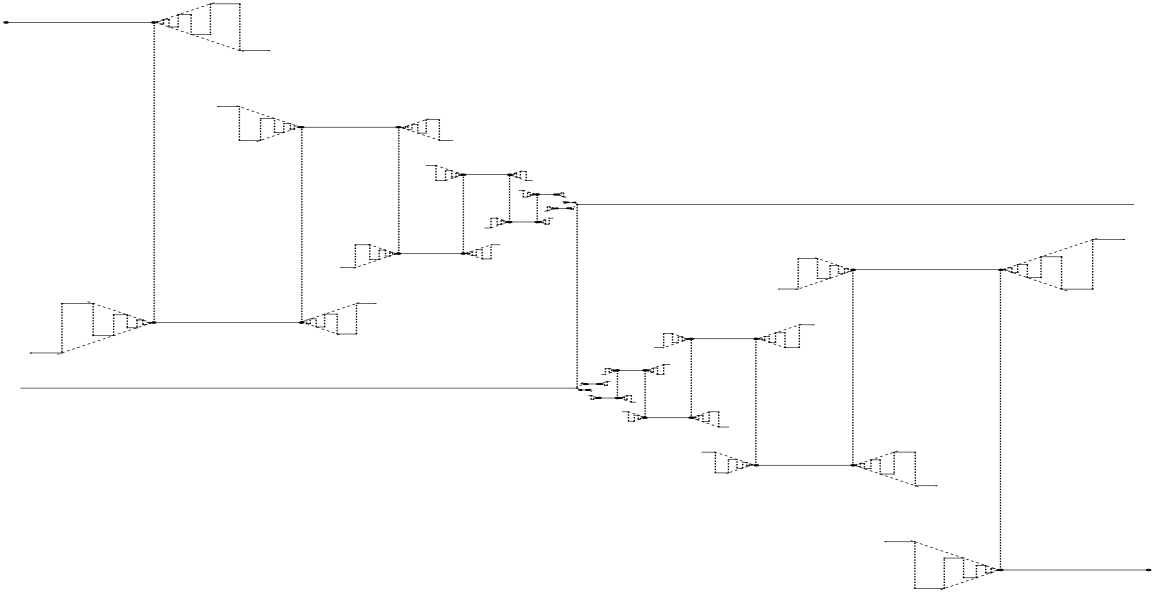}
\vspace*{-2ex}\caption{\label{f:Counter2}%
A compact total 2-fold, but not 3-fold, Henkin-continuous relation}
\end{figure}

Many properties of Henkin-continuity
translate to the iterated case:

\begin{lemma} \label{l:Henkin2}
Fix $\ell\in\IN$.
\begin{enumerate}
\item[a)]
If $f$ is $(\ell+1)$-fold Henkin-continuous, it is also
$\ell$-fold Henkin-continuous; but not necessarily vice versa.
\item[b)]
If $f:X\toto Y$ is uniformly strongly continuous
(and in particular if $f:X\to Y$ is uniformly continuous),
it is $\ell$-fold Henkin-continuous for every $\ell$.
\item[c)]
If $f:X\times Y$ is $\ell$-fold Henkin-continuous
and tightens $g:\subseteq X\times Y$,
then $g$ is $\ell$-fold Henkin-continuous 
(on $\dom(g)$) as well.
\item[d)]
If $f:X\times Y$ and $g:Y\times Z$
are both $\ell$-fold Henkin-continuous,
then so is $g\circ f$ (on $\dom(g\circ f)$).
\end{enumerate}
\end{lemma}
\begin{proof}
\begin{enumerate}
\item[a)] 
The first claim is obvious;
failure of the converse is demonstrated
in Example~\ref{x:Hierarchy}.
\item[b)]
immediate induction.
\item[c)] 
As in the proofs of
Observation~\ref{o:Composition}a),
$g$ restricts the range of the universal quantifiers
occurring in Equations~(\ref{e:Henkin2})
and extends the range of the existential quantifiers.
\item[d)]
By hypothesis we have Equation~(\ref{e:Henkin2}) for $f$ and 
the following for $g$:
\begin{multline*}
\left(\begin{array}{cc}
\forall \delta_1>0 & \exists \gamma_1>0 \\[0.5ex]
\forall y_1\in Y \; & \exists z_1\in g(y_1)
\end{array}\right) \quad
\left(\begin{array}{cc}
\forall \delta_2>0 & \exists \gamma_2>0 \\[0.5ex]
\forall y_2\in\ball(y_1,\gamma_1) \; & \exists z_2\in f(y_2)\cap\ball(z_1,\delta_1)
\end{array}\right) \;\;
\cdots \\[1ex]
\cdots\;\; \left(\begin{array}{cc}
\forall \delta_\ell>0 & \exists \gamma_\ell>0 \\[0.5ex]
\forall y_\ell\in\ball(y_{\ell-1},\gamma_{\ell-1}) \; & 
\exists z_\ell\in f(y_\ell)\cap\ball(z_{\ell-1},\delta_{\ell-1})
\end{array}\right) \\[1ex]
\forall y_{\ell+1}\in\ball(y_\ell,\gamma_\ell) \quad
\exists z_{\ell+1}\in\ball(z_\ell,\delta_\ell)\cap f(y_{\ell+1}) \enspace . 
\end{multline*}
Now inductively, to
$\varepsilon_{k+1}>0$ and to
$x_{k+1}\in\dom(g\circ f)\cap\ball(x_k,\delta_k)$,
there exist $\delta_{k+1}>0$ independent of $x_{k+1}$ and
$y_{k+1}\in f(x_{k+1})\cap\ball(y_k,\varepsilon_k)$
independent of $\varepsilon_{k+1}$;
to which in turn there exist 
$\gamma_{k+1}>0$ independent of $y_{k+1}$
and $z_{k+1}\in g(y_{k+1})\cap\ball(z_k,\delta_k)$
independent of $\delta_k$.
\qed\end{enumerate}\end{proof}
%
\subsection{Examples and Properties}
Note that $\delta_2$ in Equation~(\ref{e:Henkin2}),
although independent of $x_2$, may well depend on $x_1$:
which perhaps does not entirely express what might
be expected from a notion of uniform continuity for relations.
On the other hand, just like continuity on a compact set is in
the single-valued case equivalent to uniform continuity, we establish

\begin{lemma}  \label{l:Henkin3}
For compact $X$, total $f:X\toto Y$, and $\ell\in\IN$, the following
are equivalent:
\begin{enumerate}
\setlength{\abovedisplayskip}{2pt plus 3pt minus 3pt}
\item[i)]  $f$ is $\ell$-fold Henkin-continuous \medskip
\item[ii)] 
$\displaystyle
\left(\begin{array}{c}
\forall \varepsilon>0 \quad
\exists \delta>0 \\[0.5ex]
\forall x_1\in X \exists y_1\in f(x_1)\;\;\forall x_2\in X\exists y_2\in Y
\;\cdots\;
\forall x_\ell\exists y_\ell\;
\forall x_{\ell+1}\exists y_{\ell+1}
\end{array}\right)$ 
\begin{equation} \label{e:Henkin3}
\qquad\qquad\bigwedge\nolimits_{k=1}^\ell
\big(x_{k+1}\in\cball(x_k,\delta)\rightarrow 
  y_{k+1}\in f(x_{k+1})\cap\ball(y_k,\varepsilon)\big)
\end{equation}
\item[iii)]
There exists a total function $\lambda:\IN\to\IN$ 
such that
\begin{multline} \label{e:Henkin4}
\forall x_1 \;
\exists y_1\in f(x_1) \quad
\forall m_1\in\IN \;
\forall x_2\in\cball(x_1,2^{-\lambda(m_1)}) \;
\exists y_2\in f(x_2)\cap\cball(y_1,2^{-m_1}) \\
\forall m_2\in\IN \;\;
\forall x_3\in\cball(x_2,2^{-\lambda(m_2)}) \;\;
\exists y_3\in f(x_3)\cap\cball(y_2,2^{-m_2})\;\;\cdots\qquad\qquad \\
\cdots\;\; 
\forall m_{\ell}\in\IN \;
\forall x_{\ell+1}\in\cball(x_\ell,2^{-\lambda(m_{\ell})}) \;
\exists y_{\ell+1}\in f(x_{\ell+1})\cap\cball(y_\ell,2^{-m_{\ell}}) \enspace . 
\end{multline}
\end{enumerate}
For non-compact $X$, it still holds `i)$\Leftarrow$ii)$\Leftrightarrow$iii)''.
\end{lemma}
We call $\lambda$ as in iii) a
\textsf{modulus of $\ell$-fold Henkin-continuity of $f$}.
\begin{proof}
Note that $\delta_k$ in Equation~(\ref{e:Henkin2}) may depend
on $x_1,\ldots,x_{k-1}$; and $y_k$ on $\varepsilon_1,\ldots,\varepsilon_{k-1}$.
\begin{description}
\item[ii)$\Rightarrow$i):]
Apply Equation~(\ref{e:Henkin3}) 
to $\varepsilon:=\min\{\varepsilon_1,\ldots,\varepsilon_\ell\}$
and take $\delta_1:=\cdots=:\delta_\ell:=\delta$
in (\ref{e:Henkin2}).
\item[i)$\Rightarrow$ii):]
Recall that
$\binom{\forall\varepsilon_k\exists\delta_k}{\forall x_k\exists y_k}$
clearly implies $\forall\varepsilon_k\forall x_k\exists\delta_k\exists y_k$.
Moreover we may replace the open balls $\ball(x_k,\delta_k)$ with
their topological closures $\cball(x_k,\delta_k)$ by
reducing $\delta_k$ a bit.
Now exploit compactness and slightly extend (the proof of) Lemma~\ref{l:Implications}g) 
to see that $\delta_k$ can be chosen independent of $x_1,\ldots,x_k$,
that is, 
$\forall\varepsilon_j\forall x_k\in\cball(x_{k-1},\delta_{k-1})
\exists y_k\exists\delta_j\forall x_{k+1}\exists y_{k+1}$
implies $\forall\varepsilon_j\exists\delta_j\forall x_k
\in\cball(x_{k-1},\delta_{k-1})
\exists y_k\forall x_{k+1}\exists y_{k+1}$
for every $1\leq j\leq k\leq\ell$.
More formally, let $\Phi(\delta_j,x_k,\delta_k)$ denote the formula
\[ \exists y_k\in f(x_k)\cap\ball(y_{k-1},\varepsilon_{k-1})\quad
\forall x_{k+1}\in\cball(x_k,\delta_k)\cdots \]
Then, by hypothesis, to $\varepsilon_j>0$ and arbitrary but fixed 
$x_k\in\cball(x_{k-1},\delta_{k-1})$,
there exists $\delta_j=\delta_j(x_k)>0$
such that $\Phi(\delta_j,x_k,\delta_k)$ holds. Now by triangle inequality,
every $x_k'\in\ball(x_k,\delta_k/2)\cap\cball(x_{k-1},\delta_{k-1})$
satisfies $\Phi\big(\delta_j(x_k),x_k',\delta_k/2\big)$.
The relatively open balls $\ball(x_k,\delta_k/2)\cap\cball(x_{k-1},\delta_{k-1})$ 
cover compact $\cball(x_{k-1},\delta_{k-1})\subseteq X$, hence
finitely many of them suffice to do so. And these induce finitely many
$\delta_j(x_k)$, such that their minimum $\delta_j$ satisfies
$\Phi(\delta_j,x_k',\delta_k/2)$ for every $x_k'\in\cball(x_{k-1},\delta_{k-1})$.

Inductively swapping quantifiers as justified above, we deduce 
\begin{multline*}
\left(\begin{array}{cc}
\forall \varepsilon_1>0 \;&\;
\exists \delta_1>0 \\[0.5ex]
\forall x_1\in X &\exists y_1\in f(x_1)
\end{array}\right) \;\;
\forall\varepsilon_2>0\;\exists\delta_2>0
\;\cdots\;
\forall\varepsilon_\ell>0\;\exists\delta_\ell>0
\\[0.5ex]
\forall x_2\in\cball(x_1,\delta_1) \; \exists y_2\in f(x_2)\cap\ball(y_1,\varepsilon_1)
\;\cdots\;\qquad\qquad\qquad
\\[0.5ex] \qquad\qquad\cdots\;
\forall x_\ell\in\cball(x_{\ell-1},\delta_{\ell-1}) \; 
\exists y_\ell\in f(x_\ell)\cap\ball(y_{\ell-1},\varepsilon_{\ell-1}) \\[0.5ex]
\forall x_{\ell+1}\in\cball(x_\ell,\delta_\ell) \;
\exists y_{\ell+1}\in\ball(y_\ell,\varepsilon_\ell)\cap f(x_{\ell+1}) 
\end{multline*}
and, by one further step, obtain independence of $\delta_2,\ldots,\delta_\ell$
even from $x_1\in X$:
\begin{multline*}
\left(\begin{array}{c}
\forall \varepsilon_1
\exists \delta_1 \;
\forall\varepsilon_2\exists\delta_2
\;\cdots\;
\forall\varepsilon_\ell\exists\delta_\ell \\[0.5ex]
\forall x_1\in X \quad \exists y_1\in f(x_1)
\end{array}\right) \;\;
\forall x_2\in\cball(x_1,\delta_1) \; \exists y_2\in f(x_2)\cap\ball(y_1,\varepsilon_1)
\;\cdots\; \\[0.5ex]
\cdots\;
\forall x_\ell\in\cball(x_{\ell-1},\delta_{\ell-1}) \; 
\exists y_\ell\in f(x_\ell)\cap\ball(y_{\ell-1},\varepsilon_{\ell-1}) \\[0.5ex]
\qquad\qquad\forall x_{\ell+1}\in\cball(x_\ell,\delta_\ell) \;
\exists y_{\ell+1}\in\ball(y_\ell,\varepsilon_\ell)\cap f(x_{\ell+1})
\end{multline*}
Apply this to given $\varepsilon>0$ by choosing
let $\varepsilon_1:=\cdots=:\varepsilon_\ell:=\varepsilon$
and taking $\delta:=\min\{\delta_1,\ldots,\delta_\ell\}$.
\item[ii)$\Rightarrow$iii):]
For $m\in\IN$ set $\varepsilon:=2^{-m}$, apply ii)
to obtain some $\delta=\delta(m)$, and define $\lambda(m):=\lceil\log_2(1/\delta)\rceil$.
We show inductively that this satisfies Equation~(\ref{e:Henkin4}).
To $x_1\in X$, ii) yields some $y_1\in f(x_1)$ independent of $\varepsilon$;
now given furthermore $m_1\in\IN$, apply ii) to $\varepsilon:=2^{-m_1}$ and obtain
some $\delta>0$ (which by construction dominates $2^{-\lambda(m_1)}$)
and to every $x_2\in\ball(x_1,2^{-\lambda(m_1)})$ some $y_2\in f(x_2)\cap\ball(y_1,2^{-m_1})$;
next, to $m_2\in\IN$, ii) with $\varepsilon:=2^{-m_2}$ yields
some $\delta\geq2^{-\lambda(m_2)}$ and
to every $x_3\in\ball(x_2,2^{-\lambda(m_2)})$ some $y_3\in f(x_3)\cap\ball(y_2,2^{-m_2})$;
and so on.
\item[iii)$\Rightarrow$ii):]
To $\varepsilon>0$, take $m:=\lceil\log_2(1/\varepsilon)\rceil$ and 
$\delta:=2^{-\lambda(m)}$ with $\lambda:\IN\to\IN$ according to iii).
Then by Equation~(\ref{e:Henkin4}) inductively, to every 
$m_k:=m$ and every $x_{k+1}\in\ball(x_k,\delta)=\ball(x_k,2^{-\lambda(m_k)})$,
there exists some $y_{k+1}\in f(x_{k+1})\cap\ball(y_k,2^{-m_k})\subseteq\ball(y_k,\varepsilon)$.
\qed\end{description}\end{proof}
\begin{observation} \label{o:EquiProduct}
If the family $f_i:X_i\toto Y_i$ ($i\in I$) is $\ell$-fold Henkin-\emph{equi}continuous
in the sense of have a common modulus $\lambda$ of $\ell$-fold Henkin-continuity,
this will also be a modulus of $\ell$-fold Henkin-continuity
for $\prod_{i\in I} f_i:\prod_{i\in I} X_i\toto\prod_{i\in I} Y_i$
with respect to the maximum metrics $d\big((x_i),(x_i')\big)=\max_{i\in I} d_i(x_i,x_i')$
and $d\big((y_i),(y_i')\big)=\max_{i\in I} d_i(y_i,y_i')$.
\end{observation}
Note also that equivalence of the Cauchy representation $\myrho$
to the signed digit representation $\rhosd$ means that its inverse 
$\rhosd^{-1}:\IR\toto\Sigma^\omega$ be computable.
Hence Fact~\ref{f:VascoPeter94} asserts that $\rhosd^{-1}$ has a
strongly continuous (and w.l.o.g. pointwise compact) tightening.
We now strengthen this as well as Proposition~\ref{p:SignedDigit}c)+d):

\begin{proposition} \label{p:SignedDigit2}
\begin{enumerate}
\item[a)]
Let $x=\sum_{n=-N}^\infty a_n 2^{-n}$ be a signed digit expansion
and $k\in\IN$ such that 
$(a_n,a_{n+1})\in\{\sdone\sdzero,\sdminus\sdzero,\sdzero\sdone,\sdzero\sdminus,\sdzero\sdzero\}$
for each $n>k$.
Then every $x'\in\cball(x,2^{-k}/\mathbf{6})$ admits a signed digit expansion
$x'=\sum_{n=-N}^\infty b_n 2^{-n}$ satisfying $a_n=b_n\forall n\leq k$
\\ \underline{and} $(b_n,b_{n+1})
\in\{\sdone\sdzero,\sdminus\sdzero,\sdzero\sdone,\sdzero\sdminus,\sdzero\sdzero\}$
for all $n>k\mathbf{+1}$.
\item[b)]
Let 
\quad $\calD\;:=\; \big\{\bar\sigma\in\dom(\rhosd):\sigma_N=\sddot,
(\sigma_n,\sigma_{n+1})\in
\{\sdone\sdzero,\sdminus\sdzero,\sdzero\sdone,\sdzero\sdminus,\sdzero\sdzero\}
\;\forall n>N \big\}$.
Then $(\rhosd|_\calD)^{-1}:\IR\toto\calD$ tightens the signed digit representation
and is uniformly strongly continuous with $\delta(2^{-n-1}):=2^{-n}/6$.
\item[c)]
In particular, $\rhosd^{-1}$ is 
$\ell$-fold Henkin-continuous for every $\ell\in\IN$ with modulus $\lambda:m\mapsto m+2$.
\end{enumerate}
\end{proposition}
\begin{proof}
\begin{enumerate}
\item[a)]
First consider the case $a_{k+1}=0$. 
Then $x'':=\sum_{n=-N}^k a_n 2^{-n}=\sum_{n=-N}^{k+1} a_n 2^{-n}$ has
$0\leq x-x''\leq2^{-k}/3$ due to Proposition~\ref{p:SignedDigit}b).
Hence $x'-x''=(x'-x)+(x-x'')\in[-2^{-k}/6,2^{-k}/2]
\subseteq[-\tfrac{2}{3}\cdot2^{-k},+\tfrac{2}{3}\cdot2^{-k}]$ has,
again according to Proposition~\ref{p:SignedDigit}b),
a signed digit expansion $x'-x''=\sum_{n=k+1}^\infty b_n2^{-n}$ 
with $(b_n,b_{n+1})
\in\{\sdone\sdzero,\sdminus\sdzero,\sdzero\sdone,\sdzero\sdminus,\sdzero\sdzero\}$
for all $n$.
This yields $x'=(x'-x'')+x''=\sum_{n=-N}^k a_n 2^{-n}+\sum_{n=k+1}^\infty b_n2^{-n}$
an expansion with the claimed properties.

It remains to consider the case $a_{k+1}=\sdone$
(and $a_{k+1}=\sdminus$ proceeds analogously).
Here the hypothesis on $(a_n,a_{n+1})$ asserts
$a_{k+2}=\sdzero$. Therefore 
$x'':=\sum_{n=-N}^{k\mathbf{+1}} a_n 2^{-n}=\sum_{n=-N}^{k+2} a_n 2^{-n}$ has
$0\leq x-x''\leq2^{-k}/6$ due to Proposition~\ref{p:SignedDigit}b).
Hence $x'-x''=(x'-x)+(x-x'')\in[-2^{-k}/6,2^{-k}/3]
\subseteq[-\tfrac{2}{3}\cdot2^{-(k+1)},+\tfrac{2}{3}\cdot2^{-(k+1)}]$ has,
again according to Proposition~\ref{p:SignedDigit}b),
a signed digit expansion $x'-x''=\sum_{n=k+\mathbf{2}}^\infty b_n2^{-n}$ 
with $(b_n,b_{n+1})
\in\{\sdone\sdzero,\sdminus\sdzero,\sdzero\sdone,\sdzero\sdminus,\sdzero\sdzero\}$
for all $n$.
This yields $x'=(x'-x'')+x''=\sum_{n=-N}^{k+1} a_n 2^{-n}+\sum_{n=k+2}^\infty b_n2^{-n}$
an expansion with the claimed properties.
\item[b)]
According to a), every $x'$ admits a signed digit expansion 
$x'=\sum_{n=-N}^\infty b_n2^{-n}$ 
with $(b_n,b_{n+1})\in\{\sdone\sdzero,\sdminus\sdzero,\sdzero\sdone,\sdzero\sdminus,\sdzero\sdzero\}$,
i.e. encoding a $\rhosd$--name $\bar\sigma\in\calD$.
Morever, to each expansion $x=\sum_{n=-N}^\infty a_n2^{-n}$ 
with $(a_n,a_{n+1})\in\{\sdone\sdzero,\sdminus\sdzero,\sdzero\sdone,\sdzero\sdminus,\sdzero\sdzero\}$
corresponding to a $\rhosd$--name $\bar\sigma\in\calD$ and each $k\in\IN$,
a) asserts that also every $x'\in\cball(x,2^{-k}/6)$ admits a 
$\rhosd$--name $\bar\sigma'\in\calD\cap\cball(\bar\sigma,2^{-k-1})$:
the $-1$ arising because the digit $\sddot$ is also shared by both $\bar\sigma$ and $\bar\sigma'$.
\item[c)]
follows from b) in view of Lemma~\ref{l:Henkin2}b).
\qed\end{enumerate}\end{proof}

\subsection{Infinitary Henkin Continuity and the Main Result}

\begin{lemma} \label{l:Infinitary}
For a total, pointwise compact multifunction $f:X\toto Y$,
the following are equivalent:
\begin{enumerate}
\item[i)] $f$ admits a modulus $\lambda$ of $\ell$-fold Henkin-continuity independent of $\ell\in\IN$
\item[ii)] 
the following infinitary formula holds:
\begin{multline} \label{e:Infinitary}
\exists\delta_1,\delta_2,\cdots,\delta_\ell,\cdots >0\quad
\forall x_1\in X
\exists y_1\in Y
\forall x_2\in X
\exists y_2\in Y
\cdots
\forall x_\ell\in X
\exists y_\ell\in Y
\cdots:\\[0.5ex]
y_1\in f(x_1) \;\wedge\; 
\bigwedge\nolimits_{\ell\in\omega} \Big(
x_{\ell+1}\in\cball(x_\ell,\delta_\ell)
\;\rightarrow\; y_{\ell+1}\in f(x_{\ell+1})\cap\cball(y_\ell,2^{-\ell})
\Big) 
\end{multline}
\end{enumerate}
\end{lemma}
Naturally, Formula~(\ref{e:Infinitary}) is endowed with the semantics
of an infinite two-player game (and we make sure not to rely on determinacy). 
For a more in-depth background on 
infinitary logics, the reader may refer to \cite{Keisler,Knight}.
\begin{proof}
\begin{description}
\item[i)$\Rightarrow$ii):]
For each $m\in\IN$ let $\delta_m:=\lambda(m)$.
Now apply Equation~(\ref{e:Henkin4}) to $m_1:=1,m_2:=,\cdots,m_\ell:=\ell\cdots$:
Fix $\ell$; then, to $x_1\in X$ there exists $y_1^{(\ell)}\in f(x_1)$;
to $x_2\in\cball(x_1,\delta_1)=\cball(x_1,2^{-\lambda(m_1)})$
there exists $y_2^{(\ell)}\in f(x_2)\cap\cball(y_1,2^{-1})$;
and, inductively, to $x_{\ell+1}\in\cball(x_\ell,\delta_\ell)=\cball(x_\ell,2^{-\lambda(m_\ell)})$
there exists $y_{\ell+1}^{(\ell)}\in f(x_{\ell+1})\cap\cball(y_\ell,2^{-\ell})$.
Note that the $y_k^{(\ell)}$ indeed depend on $\ell$
since the hypothesis asserts $\lambda$ to be a modulus of
$\ell$-fold Henkin-continuity for every fixed $\ell$ only.
On the other hand, for each such $\ell$, the sequence
$(y_k^{(\ell)})_k$ `lives' in $\vartimes_k f(x_k)$;
which is compact according to \textsf{Tychonoff}:
recall our hypothesis that $f$ be pointwise compact.
Hence the sequence of sequences 
$\big((y_k^{(\ell)})_k\big)_\ell$
has a subsequence converging to some
$(y_k)_k\in\vartimes_k f(x_k)$;
and $y_{k+1}^{(\ell)}\in\ball(y_k^{(\ell)},2^{-k})$
implies $y_{k+1}\in\cball(y_k,2^{-k})$.
\item[ii)$\Rightarrow$i):]
For each $m\in\IN$ let $\lambda(m):=\lceil\log_2(1/\delta_m)\rceil$.
\\
We first assert this to be a modulus of 2-fold Henkin-continuity:
For $x_1\in X$, apply Equation~(\ref{e:Infinitary}) to $x_1=:x_1'=:x_2'=:\cdots=:x_{m_1}'$
and obtain ($y_1',\ldots,y_m'$ as well as) a $y_{m_1}'=:y_1\in f(x)$ such that
for every $x_{m_1+1}':=x_2\in\cball(x_1,2^{-\lambda(m_1)})\subseteq\cball(x_{m_1}',\delta_{m_1})$
there exists some $y_2:=y_{m_1+1}'\in f(x_2)\cap\cball(y_1,2^{-m_1})$.
\\
Now iterating this argument inductively shows $\lambda$
to be a modulus of $\ell$-fold Henkin-continuity
for every $\ell\in\IN$.
\qed\end{description}\end{proof}
Let us say that $f$ is $\omega$-fold Henkin-continuous if 
it satisfies Equation~(\ref{e:Infinitary}).
On Cantor space, this may be regarded as a uniform 
version of \emph{K\"{o}nig's Lemma}; cmp. \cite{Kohlenbach}.
And indeed we have

\begin{proposition} \label{p:Selection}
Suppose $F:\subseteq\Cantor\toto\Cantor$ 
maps compact sets to compact sets
and is $\omega$-fold Henkin-continuous.
Then $F$ admits a uniformly continuous total selection 
$G:\dom(F)\to\Cantor$. \\
More precisely if $\lambda$ is a modulus of $\ell$-fold Henkin-continuity of $F$
for every $\ell$, then $\lambda$ is also a modulus of continuity of $G$.
\end{proposition}
\begin{proof} 
Note that the triangle inequality in $\Cantor$
strengthens to
$d(\bar x,\bar z)\leq\max\{d(\bar x,\bar y),d(\bar y,\bar z)\}$.
Moreover it is no loss of generality to suppose
$\delta_\ell=2^{-\lambda(\ell)}>\delta_{\ell+1}$ 
for each $\ell$ in Equation~(\ref{e:Infinitary}).
Now with \mycite{Lemma~2.1.11.2}{Weihrauch} in mind,
we first construct a `block-monotone'
partial mapping $g:\subseteq\{0,1\}^*\to\{0,1\}^*$;
more specifically: $g:\{0,1\}^{\lambda(\ell)}\to\{0,1\}^{\ell}$
for every $\ell\in\IN$ such that
$g(\vec a)$ is (defined and) an initial substring of $g(\vec a\vec b)$
whenever $\vec a\in\{0,1\}^{\lambda(\ell)}$ and 
$\vec b\in\{0,1\}^{\lambda(\ell+1)-\lambda(\ell)}$
satisfy $\vec a\vec b\in\dom(g)$.
The construction proceeds inductively as follows:

For $\vec x_1\in\{0,1\}^{\lambda(1)}$,
consider some $\bar x_1\in\dom(F)$ extending $\vec x_1$,
i.e. $\bar x_1\in\vec x_1\circ\Cantor$.
If no such $\bar x_1$ exists, $g(\vec x_1)$ shall be undefined;
otherwise there is by hypothesis some $\bar y_1\in F(\bar x_1)$
satisfying the matrix of Equation~(\ref{e:Infinitary}):  
then define $g(\vec x_1):=\vec y_1:=\bar y_1|_{_{\leq1}}$, 
the first symbol of $\bar y_1$.
For $\vec x_2\in\vec x_1\circ\{0,1\}^{\lambda(2)-\lambda(1)}$,
if there exists some $\bar x_2\in(\vec x_2\circ\Cantor)\cap\dom(F)$,
it holds $\bar x_2\in\cball(\bar x_1,2^{-\lambda(1)})$
and we may set $g(\vec x_2):=\vec y_2:=\bar y_2|_{_{\leq2}}$ with
$\bar y_2\in F(\bar x_1)\cap(\vec y_1\circ\Cantor)$
according to Equation~(\ref{e:Infinitary}).
Inductively, for $\vec x_{\ell+1}\in\vec x_\ell\circ\{0,1\}^{\lambda(\ell+1)-\lambda(\ell)}$,
if $\emptyset\neq(\vec x_{\ell+1}\circ\Cantor)\cap\dom(F)\ni\bar x_{\ell+1}$,
set $g(\vec x_{\ell+1}):=\vec y_{\ell+1}:=\bar y_{\ell+1}|_{_{\leq\ell}}$ with
$\bar y_{\ell+1}\in F(\bar x_\ell)\cap(\vec y_\ell\circ\Cantor)$
according to Equation~(\ref{e:Infinitary}).  

Now observe that $\emptyset\neq(\vec x_{\ell+1}\circ\Cantor)\cap\dom(F)$
implies $\emptyset\neq(\vec x_{\ell}\circ\Cantor)\cap\dom(F)$;
hence, for $\bar x\in\dom(F)$,
$g(\bar x|_{\leq\lambda(\ell)})$ is defined for every $\ell$.
Since $g$ is `block-monotone' in the above sense,
$G(\vec x):=\lim_{\ell} \big(g(\bar x|_{\leq\lambda(\ell)})\circ0^\omega\big)$ 
is well-defined on $\dom(F)$; and continuous with modulus $\lambda$
via its construction through $g$. 
Moreover, $\bar y:=G(\bar x)$ satisfies by definition 
$\bar y=\lim_{\ell} \bar y_{\ell}$ with 
$\bar y_{\ell+1}\in\cball(\bar y_\ell,2^{-\ell})\cap F(\bar x_{\ell+1})$
for some $\bar x_{\ell+1}\in\cball(\bar x,2^{-\ell})$;
hence $(\bar x_{\ell},\bar y_{\ell})$ is a sequence in $F$
converging to $(\bar x,\bar y)$ with $\bar x\in\dom(F)$.
By hypothesis, $F$ maps compact $\{\bar x_{\ell}:\ell\}\cup\{\bar x\}$
to a compact set containing $\{\bar y_\ell\}$, requiring
$(\bar x,\bar y)\in F$: $G$ is a selection of $F$.
\qed\end{proof}
We can now strengthen Theorem~\ref{t:Henkin}:

\begin{theorem} \label{t:Henkin2}
Fix compact $K\subseteq\IR^d$.
\begin{enumerate}
\item[a)]
Let $f:K\toto\IR$ be computable relative to oracle $\calO$.
Then there exists $g:K\toto\IR$ tightening
$f$ which is still computable relative to $\calO$
and maps compact sets to compact sets.
\item[b)]
If $f:K\toto\IR$ is relatively computable,
it is $\omega$-fold Henkin-continuous.
\item[c)]
Suppose $f:K\toto\IR$
maps compact sets to compact sets
and is $\omega$-fold Henkin-continuous.
Then $f$ is relatively computable.
\end{enumerate}
\end{theorem}
This theorem provides the
desired topological characterization of relative computability:

\begin{corollary} \label{c:Main}
For $X:=[0,1]^d$, a total relation $f:X\toto\IR$ 
mapping compact sets to compact sets
(and in particular one with compact graph)
is relatively computable
~iff~ it satisfies Equation~(\ref{e:Infinitary}).
\end{corollary}

\begin{proof}[Theorem~\ref{t:Henkin2}]
\begin{enumerate}
\item[a)]
By hypothesis, $f$ admits an $\calO$-computable
(and thus continuous)
$(\rhosd^d,\rhosd)$--realizer 
$F:\subseteq\Cantor\to\Cantor$ 
on compact $\dom(F)=\dom\big(\rhosd^d\big|^K\big)$, 
i.e. mapping compact sets to compact sets.
And so does $(\rhosd^d){-1}$ (Example~\ref{x:ClosedMap}b)
and continuous $\rhosd$. Thus, again according to 
Lemma~\ref{l:Tighten}d), also 
$g:=\rhosd\circ F\circ(\rhosd^d)^{-1}:K\toto\IR$
maps compact sets to compact sets;
and tightens $f$ (Lemma~\ref{l:Tighten}f); 
and is computable relative to $\calO$.
\item[b)]
According to a) and Lemma~\ref{l:Henkin2}c) we may w.l.o.g. suppose that 
$f$ maps compact sets to compact sets and
in particular that $C:=f[K]$ is compact.
Combining Proposition~\ref{p:SignedDigit2}c) with Observation~\ref{o:EquiProduct}
and Example~\ref{x:Representations}f) shows 
$(\rhosd^d)^{-1}:\IR^d\toto\Cantor$ to be $\omega$-fold Henkin-continuous.
By hypothesis, $f$ admits a continuous $(\rhosd^d,\rhosd)$--realizer 
$F:\subseteq\Cantor\to\Cantor$ 
on compact $\dom(F)=\dom(\rhosd^d\big|^K)$; in particular,
$F$ is uniformly continuous. Moreover, 
$\rhosd|^C\circ F\circ\big(\rhosd^d\big|^K\big)^{-1}:K\toto C\subseteq\IR$
tightens $f$ (Lemma~\ref{l:Tighten}f) with $\dom(\rhosd|^C)$ compact,
hence $\rhosd|^C:\subseteq\{0,1\}\to C$ is uniformly continuous.
Now apply Lemma~\ref{l:Henkin2}b)+c)+d)
to conclude that both $\rhosd|^C\circ F\circ(\rhosd^d)^{-1}$
and $f$ are $\omega$-fold Henkin-continuous.
\item[c)]
As in the proof of Theorem~\ref{t:Henkin}c),
observe that $F:=\rhosd^{-1}\circ f\circ\rhosd^d\big|^K$
is $\omega$-fold Henkin-continuous
according to Proposition~\ref{p:SignedDigit2}c) and
Lemma~\ref{l:Henkin2}b)+c)+d).
And $F$ maps compact sets to compact sets
(Lemma~\ref{l:Tighten}d).
Hence $F$ admits a continuous selection
$G$ on $\dom(F)=\dom\big(\rhosd^d\big|^K\big)$ due to Proposition~\ref{p:Selection}.
This is a continuous (and hence relatively computable)
$(\rhosd^d,\rhosd)$--realizer of $f$.
\qed\end{enumerate}\end{proof}

\section{Conclusion}
We have 
proposed a hierarchy of notions of uniform continuity for real relations
based on the Henkin quantifier;
and shown its $\omega$-th level to characterize 
relative computability in the compact case.

Our condition may be considered descriptionally 
simpler than the previous characterization from \cite{VascoPeter94}.
Indeed, although Equation~(\ref{e:Infinitary}) does employ countably infinitary logic,
Fact~\ref{f:VascoPeter94} even quantifies over subsets of \emph{un}countable $\IR$.

\begin{myquestion}
Does Theorem~\ref{t:Henkin2} extend from compact
subsets $K$ of $\IR^d$ to general compact metric spaces?
\end{myquestion}
A promising candidate replacement for $\rhosd^d\big|^K$ is provided in
\mycite{Proposition~4.1}{deBrecht}. But is its inverse
$\omega$-fold Henkin-continuous (or does even admit a uniformly
strongly continuous tightening) ?

\paragraph{Acknowledgements:}
The last author is grateful to \textsc{Ulrich Kohlenbach}
for pointing out that already \textsc{M.J.~Beeson} had observed
the relevance of the Henkin quantifier to continuity in constructive
mathematics; and to \textsc{Klaus Weihrauch} for providing the
`right' notion of composition for relations.


\end{document}